\documentclass[12pt]{extarticle}
\usepackage{amsmath, amsthm, amssymb, color}
\usepackage[colorlinks=true,linkcolor=blue,urlcolor=blue]{hyperref}
\usepackage{graphicx}
\usepackage{caption}
\usepackage{mathtools}
\usepackage{enumerate}
\usepackage{verbatim}
\usepackage{tikz,tikz-cd,tikz-3dplot}
\usepackage{amssymb}
\usetikzlibrary{matrix}
\usetikzlibrary{arrows}
\usetikzlibrary{positioning}
\usepackage{algorithm}
\usepackage[noend]{algpseudocode}
\usepackage{caption}
\usepackage[normalem]{ulem}
\usepackage{subcaption}
\oddsidemargin \evensidemargin
\usepackage{makecell}
\usepackage{array}
\usepackage{pgfplots}
\usepgfplotslibrary{groupplots}
\pgfplotsset{compat=1.15}
\usepackage{mathrsfs}
\usetikzlibrary{arrows}

\usepackage{macros}

\newtheorem{theorem}{Theorem}
\newtheorem{proposition}[theorem]{Proposition}
\newtheorem{lemma}[theorem]{Lemma}
\newtheorem{corollary}[theorem]{Corollary}

\theoremstyle{definition}
\newtheorem{definition}[theorem]{Definition}

\newtheorem{remark}[theorem]{Remark}
\newtheorem{conjecture}[theorem]{Conjecture}

\newtheorem{example}[theorem]{Example}

\numberwithin{theorem}{section}

\newcommand{\PP}{\mathbb{P}}
\newcommand{\RR}{\mathbb{R}}

\newcommand{\CC}{\mathbb{C}}
\newcommand{\ZZ}{\mathbb{Z}}
\newcommand{\NN}{\mathbb{N}}

\newcommand{\Disc}[1]{\ensuremath{\mathrm{Disc} \left(#1\right)}}

\DeclareMathOperator{\mldeg}{mldeg}
\DeclareMathOperator{\conv}{conv}

\title{The Euler Stratification for $\PP^1 \times \PP^1 \times \PP^n$}
\author{Serkan Ho\c{s}ten, Vadym Kurylenko, Elke Neuhaus and Nikolas Rieke}
\date{}

\begin{document}
\maketitle

\begin{abstract}
We study the Euler characteristic of a hypersurface in $(\CC^*)^2 \times (\CC^*)^n$ defined by a polynomial whose monomial support corresponds to lattice points in $\Delta_1 \times \Delta_1 \times \Delta_n$ as the coefficients of the defining polynomial vary. Each member of this hypersurface family corresponds to a three-way independence model from algebraic statistics, and the (signed) Euler characteristic is equal to the maximum likelihood degree (ML degree) of the model. We show in the case of $\Delta_1 \times \Delta_1 \times \Delta_1$ this Euler characteristic depends only on the vanishing patterns of the factors of the principal $A$-determinant, but this fails for $\Delta_1 \times \Delta_1 \times \Delta_n$ with $n \geq 2$. We prove that, for all $n\geq 1$,  all positive integers up to the maximum possible ML degree can be realized as the Euler characteristic. Furthermore, we completely determine the Euler stratification for $\PP^1 \times \PP^1 \times \PP^1$ and provide partial information for $\PP^1 \times \PP^1 \times \PP^2$. 
\end{abstract}

%\noindent Keywords: Maximum likelihood degree, principal $A$-determinant, Euler characteristic \\
%MSC2020: 14F45, 13P15, 05E14, 62R01

\section{Introduction}
In this paper, we consider the hypersurface $Y_{W,n} \subset (\CC^*)^{2+n}$ defined by 
\begin{equation} \label{eq:f}
f_W =  f_0  + \sum_{k=1}^n z_kf_k =  (w_{000} + w_{100}x_1 + w_{010}y_1 + w_{110} x_1y_1)  + \sum_{k=1}^n z_k (w_{00k} + w_{10k}x_1 + w_{01k}y_1 + w_{11k} x_1y_1)
\end{equation}
where $W = (w_{ijk}) \in (\CC^2 \otimes \CC^2 \otimes \CC^{n+1})^* = (\CC^*)^{4n+4}$.
Our goal is to understand the variation of the Euler characteristic $\chi(Y_{W,n})$ with respect to $W$. The hypersurface $Y_{W,n}$ goes with the \emph{scaled toric variety} $X_{W,n} \subset \PP^{4n+3}$ defined as the scaled Segre embedding of $\PP^1 \times \PP^1 \times \PP^n$ 
\begin{equation}
\label{eq:segre-map}
[x_0 \, : x_1] \times [y_0 \, : y_1] \times [z_0 \, : \, \cdots \, : z_n] \longrightarrow [w_{ijk} x_iy_jz_k \, : \, i,j=0,1, \, k=0,1,\ldots, n]
\end{equation}
into $\PP^{4n+3}$; see \cite[Section 2.1]{amendola2019maximum}. The usual Segre embedding of $\PP^1 \times \PP^1 \times \PP^n$ is obtained with $w_{ijk} = 1$ for all $i,j,k$, and we denote this toric variety by $\mathcal X_n$. The exponent vectors of the monomials in \eqref{eq:segre-map} are the lattice points of $P = \Delta_1 \times \Delta_1 \times \Delta_n$ where $\Delta_r$ is the convex hull of the standard unit vectors in $\RR^{r+1}$. We collect these  lattice points in the $(n+5) \times (4n+4)$ matrix $A$.

The \emph{maximum likelihood estimation problem} is an optimization problem over the positive real part of $X_{W,n}$:

\begin{equation} \label{eq:ml-problem}
\mathrm{maximize} \,\, \frac{\prod_{i,j,k} p_{ijk}^{u_{ijk}}}{\sum_{i,j,k} p_{ijk}} \,\,\,\, \text{subject to} \,\,\, (p_{ijk}) \in X_{W,n} \cap \RR_+^{4n+4}.
\end{equation}
Here $(u_{ijk}) \in \NN^{4n+4}$ is the data, and the variety $X_{W,n}$ is a \emph{three-way independence model} corresponding to two binary and one $(n+1)$-ary random variables that are probabilistically independent \cite[Chapter 9]{Sul18}.
The \emph{maximum likelihood degree} of $X_{W,n}$, $\mldeg(X_{W,n})$,  is defined to be the number of complex critical points of the rational objective function in \eqref{eq:ml-problem} over $X_{W,n}$ for generic data; see \cite{CHKS06,HKS05, HS14}. 

Our starting point is based on the following two results. The first one is a corollary of \cite[Theorem 1]{H13} or \cite[Theorem 1.7]{HS14}. For this let 
$$
\mathcal H = \left\{(p_{ijk}) \in \PP^{4n+3} \colon \left(\prod_{i,j,k} p_{ijk}\right)\left(\sum_{i,j,k} p_{ijk}\right) = 0\right\}.
$$
\begin{proposition}\label{prop:ML degree = Euler char}
The ML degree of $X_{W,n}$ is equal to the signed Euler characteristic of $Y_{W,n}$:
$$\mldeg(X_{W,n}) \, = \, (-1)^{n+2} \chi(X_{W,n} \backslash \mathcal H) =  (-1)^{n+1} \chi(Y_{W,n}).$$
\end{proposition}
We note that the second equality follows from the fact that $X_{W,n} \backslash \mathcal H$ and $(\CC^*)^{n+2} \backslash Y_{W,n}$ are isomorphic essentially via the map in \eqref{eq:segre-map} and $\chi(Y_{W,n}) = -\chi((\CC^*)^{n+2} \backslash Y_{W,n})$; see \cite[Section 6.4]{TW24} for further details. The second result is Theorem 13 of \cite{amendola2019maximum} applied to our case.
\begin{proposition} \label{prop:deficient-mldeg}
    The ML degree of $X_{W,n}$ is at most the degree of the toric variety $\mathcal X_n$. Moreover, $\mldeg(X_{W,n}) < \deg(\mathcal X_n)$ if and only if $E_A(W) = 0$ where $E_A$ is the principal $A$-determinant.
\end{proposition} 
The principal $A$-determinant above  \cite[Chapter 10, Theorem 1.2]{GKZ94} is a polynomial in the coefficients $w_{ijk}$ with a known factorization
\begin{align*}
        E_A = \prod_{\Gamma \text{ a face of } P} \Disc{\Gamma \cap \ZZ^n},
    \end{align*}
where $\Disc{\Gamma \cap \ZZ^n}$ is the $A'$-discriminant \cite[Chapter 9]{GKZ94}. Here, $A'$ is the matrix with columns corresponding to vertices in the face $\Gamma$. The factors of $E_A$ in the case of our toric variety $\mathcal X_n$ are known explicitly, see \cite[Section 2.4]{CHKO24}. We will give more details about the principal $A$-determinant and its discriminantal factors in Section \ref{sec:factors}.

The Euler characteristic $\chi(Y_{
W,n})$ and hence the ML degree $\mldeg(X_{W,n})$ is controlled by the \emph{Euler stratification} for $Y_{W,n}$ (or, by abuse of language, for $\PP^1 \times \PP^1 \times \PP^n$) \cite{TW24}: the coefficient space where $W$ lives can be stratified according to $\chi(Y_{W,n})$. A fundamental question about this stratification is summarized in the following.
\begin{conjecture} \cite[Conjecture 2.29]{CHKO24}  \label{conj:main-conjecture}
If exactly the same factors of the principal $A$-determinant $E_A$ vanish on $W$ and $W'$ then $\chi(Y_{W,n}) = \chi(Y_{W',n})$.
\end{conjecture}
This conjecture has a positive answer for the Segre product $\PP^m \times \PP^n$.
\begin{theorem} \cite[Theorem 1.3]{CHKO24} \label{thm:Pm x Pn} Let $Z_{W,m,n} \subset (\CC^*)^m \times (\CC^*)^n$ be the hypersurface defined by 
$$ g \, = \, g_0 +  \sum_{j=1}^n y_jg_j \ = \left( w_{00} + \sum_{i=1}^m w_{i0} x_i \right) + \sum_{j=1}^n y_j \left( w_{0j} +  \sum_{i=1}^m w_{ij} x_i\right)$$
where $W = (w_{ij}) \in (\CC^{m+1} \otimes \CC^{n+1})^* = (\CC^*)^{mn+m+n+1}$. Then $\chi(Z_{W,m,n}) = \chi(Z_{W',m,n})$ if exactly the same factors of the principal $A$-determinant $E_A$ vanish on $W$ and $W'$ where $A$ consists of the lattice points in $\Delta_m \times \Delta_n$. Equivalently, if the matroids defined by the matrices $[I_{m+1} \, W]$ and $[I_{m+1} \, W']$ are isomorphic, then $\chi(Z_{W,m,n}) = \chi(Z_{W',m,n})$.
\end{theorem}
\begin{remark}
Conjecture \ref{conj:main-conjecture} is false for other families of scaled toric varieties. For instance, let
$$ f = w_0 + w_1 x + w_2 x^2 + w_3 x^3$$
where $ W= (w_0, w_1, w_2, w_3) \in (\CC^*)^4$. The corresponding toric variety is the twisted cubic in $\PP^3$ and we let $Z_W$ be the hypersurface in $\CC^*$ defined by $f$. The principal $A$-determinant $E_A$  where 
$$ A = \begin{pmatrix} 3 & 2 & 1 & 0 \\ 0 & 1 & 2 & 3\end{pmatrix}$$
has a single relevant factor, namely, the classical discriminant $D_f$ of the cubic polynomial $f$. The Euler characteristic $\chi(Z_W) = 3$ if and only if $D_f(W) \neq 0$. However, $D_f(W) = D_f(W') = 0$ does not necessarily imply $\chi(Z_W)= \chi(Z_{W'})$: The discriminant vanishes if and only if $f_W(x)$ is singular in which case $\chi(Z_W) < 3$, but the Euler characteristic is equal to the number of distinct roots of $f_W(x)$, i.e. it can be equal to $1$ or $2$. 
\end{remark}
\subsection{Results}
Our first result settles Conjecture \ref{conj:main-conjecture} for $\PP^1 \times \PP^1 \times \PP^n$. In general, the statement that the vanishing of certain factors of the principal $A$-determinant determines $\chi(Y_{W,n})$ and with it the ML degree of $X_{W,n}$ is false. We provide a counterexample in
Example \ref{ex: counterexample}. This also settles the question for general independence models. For $\PP^1 \times \PP^1 \times \PP^n$, one can however say the following:
\begin{theorem} \label{thm:main}
Let $f_W$ be as in \eqref{eq:f} and let 
$Y_{W,n}$ be the corresponding hypersurface in $(\CC^*)^{n+2}$.
\begin{enumerate} 
\item $n=1$: If exactly the same factors of the principal $A$-determinant $E_A$ vanish on $W$ and $W'$, then $\chi(Y_{W,n}) = \chi(Y_{W',n})$. 
\item $n=2$: The same statement holds as long as the factor of $E_A$ corresponding to the $2\times 2 \times  3$ hyperdeterminant does not vanish or as long as one of the factors corresponding to a $2 \times 2 \times 2$ hyperdeterminant does vanish. 
\item $n > 2$: The same statement holds as long as the factors of $E_A$ corresponding to $2\times 2 \times  3$ hyperdeterminants do not vanish. 
\end{enumerate}
\end{theorem}
We will prove this theorem by first computing $\mldeg(X_{W,n})$ as in the proof of Theorem \ref{thm:Pm x Pn} from \cite{CHKO24}. In particular, we will analyze the Euler characteristic of the possible intersections of the plane quadrics defined by $f_0, f_1, \ldots, f_n$ in \eqref{eq:f}.

Despite this theorem, it is difficult to compute the Euler stratification of $Y_{W,n}$ itself. 
Nevertheless we will determine the complete Euler stratification for $\PP^1 \times \PP^1 \times \PP^1$.  As we will state later, the principal $A$-determinant for $\PP^1 \times \PP^1 \times \PP^1$ consists of seven factors one of which is a polynomial of degree $4$ called the $2 \times 2 \times 2$ hyperdeterminant.  
\begin{theorem} \label{thm:P1 x P1 x P1}
 The Euler stratification for $\PP^1 \times \PP^1 \times \PP^1$ consists of $41$ strata:
 \begin{itemize}
 \item[0.] One stratum where no factor of $E_A$ vanishes with $\chi(Y_{W,1})=6$ on this stratum. 
 \item[1.] $7$ strata where exactly one factor of $E_A$ vanishes with $\chi(Y_{W,1})=5$.
 \item[2.] $\binom{7}{2}$ strata where exactly two factors of $E_A$ vanish
 with $\chi(Y_{W,1})=4$.
 \item[3.] $8$ strata where three certain factors not containing the $2 \times 2 \times 2$ hyperdeterminant vanish with $\chi(Y_{W,1})=3$.
 \item[4.] $3$ strata where four certain factors and the $2 \times 2 \times 2$ hyperdeterminant vanish with $\chi(Y_{W,1}) =2$.
 \item[5.] One stratum where all factors of $E_A$
 vanish with $\chi(Y_{W,1}) =1$.
 \end{itemize}
\end{theorem}

Finally, we will turn to the question of the realizability of the ML degrees. We know that $\mldeg(X_{W,n}) \leq \deg(\mathcal X_n)$ and with the choice of 
$w_{ijk} = 1$ for all $i,j=0,1$ and $k=0,1,\ldots,n$ we get $\mldeg(X_{W,n}) = 1$ (this is the ML degree of the usual toric variety $\mathcal X_n = \PP^1 \times \PP^1 \times \PP^n$). Hence, the question arises whether all integers between $1$ and $\deg(\mathcal X_n)$ can be realized as
$\mldeg(X_{W,n})$. This question has an answer to the affirmative for $\PP^m \times \PP^n$ for $m=1,2,3$ \cite[Theorem 1.4]{CHKO24}. We settle it also for $\PP^1 \times \PP^1 \times \PP^n$. 
\begin{theorem} \label{thm:realizability}
For each integer $1 \leq r \leq \deg(\mathcal X_n) = (n+1)(n+2)$ there exists $W \in (\CC^*)^{4n+4}$ such that $\mldeg(X_{W,n}) = r$. 
\end{theorem}

Now we give an outline of our paper. Section \ref{sec:factors} reviews the relevant background on $A$-discriminants and the principal $A$-determinant $E_A$ for toric varieties, and then specializes to the Segre product
$\PP^1\times\PP^1\times\PP^n$.
In Proposition \ref{prop:factors 2x2x2}, we describe exhaustively all possible ways the factors of $E_A$ can vanish for $n=1$. We will also give a partial description for $n=2$. Our main goal in Section \ref{sec: three-way} is to prove Theorem \ref{thm:main}. This requires a detailed study of the Euler characteristic of the intersection of quadrics in $\PP^1 \times \PP^1$ associated
to $f_0, f_1, \ldots, f_n$ as in \eqref{eq:f}. Along the way we give a formula for the ML degree of the scaled Segre product of $\PP^m \times \PP^n$ (Corollary \ref{cor: ML-product-of-two}). At the end of the section we provide a counterexample to Conjecture \ref{conj:main-conjecture}.  Example \ref{ex: counterexample} will present
two scaling tensors $W$ and $W'$
that lead to two different ML degrees for $n\geq 2$ although exactly the same set of factors
of $E_A$ vanish on $W$ and $W'$. Section \ref{sec:P1 x P1 x P1} gives the complete Euler stratification for $\PP^1 \times \PP^1 \times \PP^1$ (Theorem \ref{thm:P1 x P1 x P1}). We achieve this by carefully analyzing the intersection patterns of two quadrics in $(\CC^*)^2$. Finally, in Section \ref{sec:realizability}, we prove Theorem \ref{thm:realizability} by exhibiting a procedure to construct scaling tensors that produce all possible values for the ML degree of scaled $\PP^1 \times \PP^1 \times \PP^n$.

\section{Factors of the Principal A-Determinant} 
\label{sec:factors}

The principal $A$-determinant $E_A$ plays a crucial role for the Euler stratification of hypersurfaces such as $Y_{W,n}$ defined by \eqref{eq:f}(equivalently, for the ML degree stratification of scaled toric varieties such as $X_{W,n})$. In this section, we will define $E_A$ and then study it in the case of 
$\PP^1 \times \PP^1 \times \PP^n$ for $n=1,2$. We start with the definition of the $A$-discriminant; see \cite[Chapter 9]{GKZ94}. 
\begin{definition}
     Let $ f_w(z) = \sum_{i = 1}^N w_iz^{a_i} $ be a polynomial in $\CC[z_1,\ldots, z_d]$ where $w=(w_1, \ldots, w_N) \in (\CC^*)^N$, and let $ A = \left[a_1 \, a_2 \, \cdots \, a_N \right] \in \ZZ^{d \times N} $.  Then 

    \begin{align*}
        \nabla_{A} = \overline{\{ w \in (\CC^*)^N : \exists  z \in (\CC^*)^d \text{      such that  } f_w(z) = \frac{\partial f_w}{\partial z_1} = \cdots = \frac{\partial f_w}{\partial z_d} = 0 \}} 
         \subset \CC^N
    \end{align*}
     parametrizes hypersurfaces $\{f_w = 0\}$ that have singular points in $(\CC^*)^d$. Typically,  
     $ \nabla_A $ is an irreducible hypersurface and its defining polynomial in $\ZZ[w_1, \ldots, w_N]$ is called the $A$-discriminant which we denote by $\Disc{A} $. If $ \nabla_A $ is not a hypersurface we set $ \Disc{A} = 1 $, in which case we say $ \Disc{A} $ is trivial.
\end{definition}

%maybe definition of principal $A$-determinant  again (it's already in the Introduction!) 

\begin{definition}
 Let $A \in \ZZ^{d \times N}$ as above and let $X_A \subset \PP^{N-1}$ be the corresponding toric variety. If $X_A$ is smooth, the principal $A$-determinant is
 \begin{equation*} \label{principleAdeterminant}
        E_A(w) = \prod_{\Gamma \text{ a face of } P} \Disc{\Gamma \cap \Z^d},
    \end{equation*}
where $ P = \conv(a_1,\ldots,a_N) $ and $ \Disc{\Gamma \cap \Z^d} $ is the $A'$-discriminant, where $A'$ is the matrix with columns corresponding to vertices in the face $\Gamma$ \cite[Chapter 10, Theorem 1.2]{GKZ94}. 
\end{definition}
\begin{remark}
In the definition above, every vertex of $P$ contributes a factor $w_i$ to the principal $A$-determinant. In this paper, we are interested in the vanishing of $E_A(w)$ in $(\CC^*)^N$, and since $w \in (\CC^*)^N$, we will ignore these factors corresponding to the vertices.    
\end{remark}
Next we describe the principal $A$-determinant induced by the hypersurface $Y_{W,n}$. Here, the matrix $A$ is of format $(n+5) \times (4n+4)$,
where the columns correspond to the 
vertices of $P~=~\Delta_1~\times~\Delta_1~\times~\Delta_n$ or, equivalently, to the exponents of the terms in  $f_W$ in \eqref{eq:f}.  

First we introduce some notation. Note that the coefficient vector  $W$ can be viewed as a $2 \times 2 \times (n+1)$ tensor with entries in $\CC^*$.  We denote by $W_{\bullet \bullet k}$ the $2 \times 2$ matrix
$$W_{\bullet \bullet k} := \begin{pmatrix}
    w_{00k} & w_{01k} \\
    w_{10k} & w_{11k}
\end{pmatrix}.
$$
Similarly, we define
$$W_{i \bullet (k_1,k_2)} := \begin{pmatrix}
    w_{i0k_1} & w_{i1k_1} \\
    w_{i0k_2} & w_{i1k_2}
\end{pmatrix}
\quad \mbox{and} \quad
W_{\bullet j (k_1,k_2)} := \begin{pmatrix}
    w_{0jk_1} & w_{1jk_1} \\
    w_{0jk_2} & w_{1jk_2}
\end{pmatrix}.$$
The first matrix is a \emph{slice} of $W$ and there are $n+1$ such slices. The next two are $2 \times 2$ submatrices of the four \emph{faces} 
$W_{0 \bullet \bullet}, W_{1 \bullet \bullet}, W_{\bullet 0 \bullet}$, and $W_{\bullet 1 \bullet}$ which are $2 \times (n+1)$ matrices. There are a total of $4 \binom{n+1}{2}$ such matrices. We will call the determinants of all these $2 \times 2$ matrices the \emph{$2$-minors} of $W$ and denote them by 
%$F_{\bullet \bullet k} := \det W_{\bullet \bullet k}$, $F_{i \bullet (k_1,k_2)} := \det W_{i \bullet (k_1,k_2)}$ and $F_{\bullet j (k_1,k_2)} := \det W_{\bullet j (k_1,k_2)}$. 
$$F_{\bullet \bullet k} := \det W_{\bullet \bullet k}, \quad F_{i \bullet (k_1,k_2)} := \det W_{i \bullet (k_1,k_2)}, \quad F_{\bullet j (k_1,k_2)} := \det W_{\bullet j (k_1,k_2)}.$$ 
Furthermore, we will use subtensors of formats $2 \times 2 \times 2$ and $2 \times 2 \times 3$. The first kind is indexed by $0 \leq k_1 < k_2 \leq n$ and we denote it by $W^{k_1 k_2}$. The second kind is indexed by $0 \leq k_1 < k_2 < k_3 \leq n$ and we denote it by $W^{k_1 k_2 k_3}$. 
There are $\binom{n+1}{2}$ subtensors of the first kind and 
$\binom{n+1}{3}$ subtensors of the 
second kind.  

\begin{definition}
The $2 \times 2 \times 2$ hyperdeterminant $H_{k_1 k_2}$ of the subtensor $W^{k_1 k_2}$
is  
$$H_{k_1 k_2} = [(w_{00k_1}w_{11k_2} - w_{00k_2}w_{11k_1}) - (w_{01k_1}w_{10k_2} - w_{01k_2}w_{10k_1})]^2 - 4 \det(W_{0 \bullet (k_1,k_2)}) \det(W_{1 \bullet (k_1,k_2)}).$$
The $2 \times 2 \times 3$ hyperdeterminant $H_{k_1 k_2 k_3}$ of the subtensor $W^{k_1 k_2 k_3}$ is equal to the resultant $\mathrm{Res}(q_{k_1}, q_{k_2}, q_{k_3})$  where 
$$q_{k} = w_{00k}x_0y_0 + w_{01k}x_0y_1 + w_{10k}x_1 y_0 + w_{11k} x_1y_1. $$
%\vadym{Must be careful here, usually resultant is for $k+1$ equations in $k$ variables. We have $3$ equation in $4$ variables (although after dehomogenization it becomes in two variables)}
The $2 \times 2 \times 2$ hyperdeterminants $H_{k_1k_2}$ are polynomials of degree four in the entries of $W^{k_1k_2}$ which have $12$ terms. The $2 \times 2 \times 3$ hyperdeterminants $H_{k_1 k_2 k_3}$ are polynomials of degree six in the entries of $W^{k_1 k_2 k_3}$ which have $66$ terms. See \cite[Chapter 14]{GKZ94} and \cite{CHKO24}. 
\end{definition}

\begin{proposition} \cite[Theorem 1.5]{CHKO24}
The principal $A$-determinant $E_A$ for $\PP^1 \times \PP^1 \times \PP^n$ is the product of all $2$-minors of $W$ together with all $2 \times 2 \times 2$ hyperdeterminants $H_{k_1k_2}$ and all $2 \times 2 \times 3$ hyperdeterminants $H_{k_1k_2k_3}$.  
\label{prop:factors_principal_A_det}
\end{proposition}

Throughout this paper, we represent subsets of the factors of the principal $A$-determinant of $\PP^1 \times \PP^1 \times \PP^n$ by unions of faces  of a tower of $n$ cubes. Each $2$-dimensional face of any cube as well as the union of two horizontal edges on the four faces $W_{0 \bullet \bullet}, W_{1 \bullet \bullet}, W_{\bullet 0 \bullet}$, and $W_{\bullet 1 \bullet}$
corresponds to a $2$-minor of $W$.
%Particularly, the tops and bottoms of each cube correspond to the determinant of the matrix $W_{\bullet \bullet k}$ and $W_{\bullet \bullet (k+1)}$ respectively.
The ``cube'' obtained by
a choice of any two slices $W_{\bullet \bullet i}$ and $W_{\bullet \bullet j}$ corresponds to a $2 \times 2 \times 2$ hyperdeterminant factor of $E_A$.  
We give examples of this in Figure \ref{fig:222_Hdets} below.

\begin{figure}[H]
    \centering
    \includegraphics[scale=1.1]{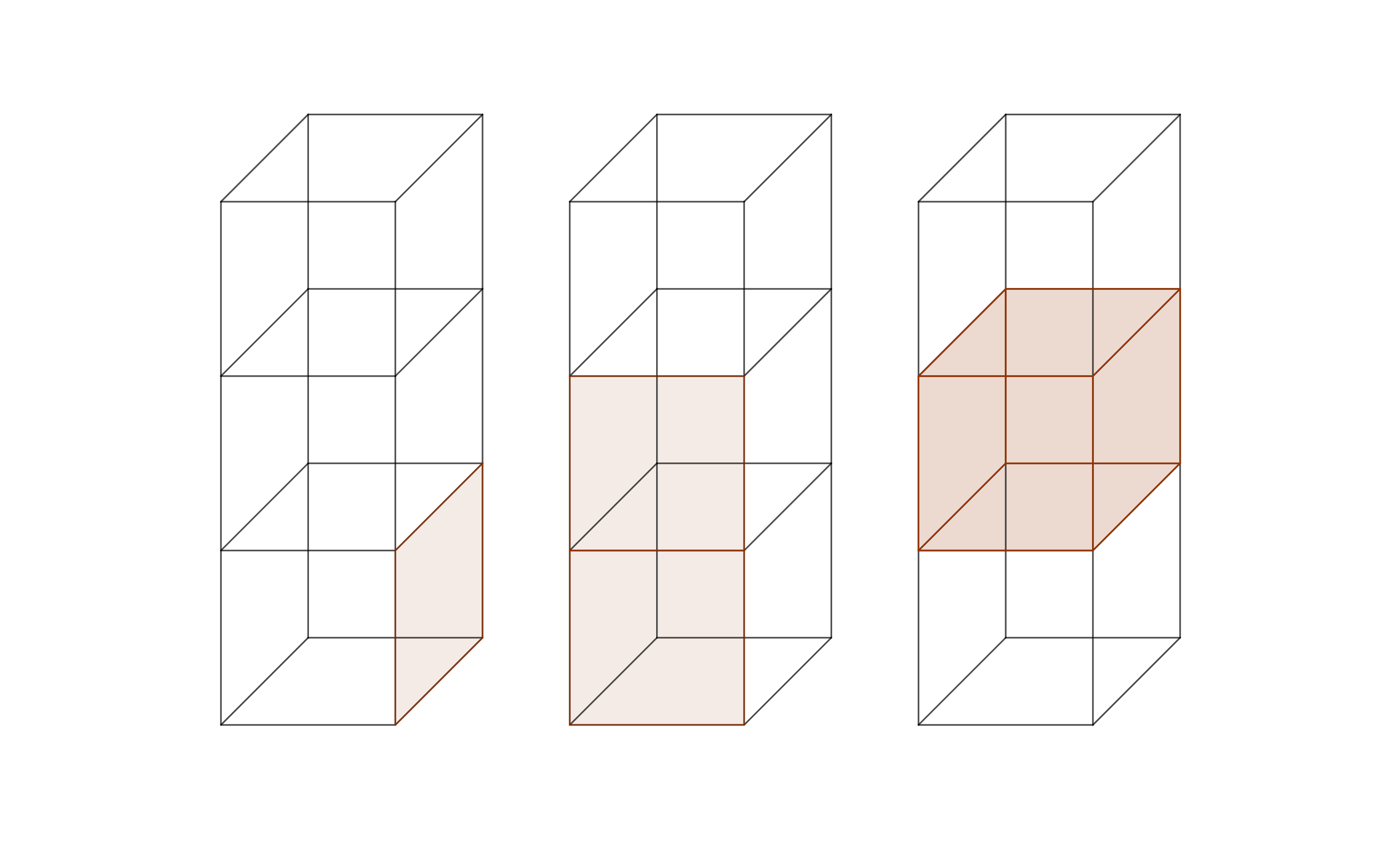}
    \caption{Two different $2 \times 2$ determinant factors and one $2 \times 2 \times 2$ hyperdeterminant factor of $E_A$ for $n = 2$}
    \label{fig:222_Hdets}
\end{figure}

As our main result in this paper we settle Conjecture \ref{conj:main-conjecture}. One question this conjecture raises is which subsets of the factors of the principal $A$-determinant can simultaneously vanish while the rest does not. We answer this question for $\PP^1 \times \PP^1 \times \PP^1$ and give some partial results for $\PP^1 \times \PP^1 \times \PP^2$.
%Now we return to the question of understanding the maximum likelihood degree of $X_{W,n}$.
%We remind the reader of the following fact:

%\begin{proposition}
%    The ML degree of $X_{W,n}$ is at most the degree of the toric variety $\mathcal X_n$. Moreover, $\mldeg(X_{W,n}) < \deg(\mathcal X_n)$ if and only if $E_A(W) = 0$ where $E_A$ is the principal $A$-determinant.
%    \label{prop:ML_degree_bounded_by_degree}
%\end{proposition} 

%We will see later on that the degree to which $\mldeg(X_{W,n})$ is less than $\deg(\mathcal X_n)$ is decided purely by which factors of $E_A$ vanish at $W$.
%So it becomes useful to understand fully when some subset of the factors of $E_A$ vanishing causes some other factor of $E_A$ to vanish.
%Below we provide a diagram which describes all such relations.

\subsection*{$\PP^1 \times \PP^1 \times \PP^1$}
In the case of $\PP^1 \times \PP^1 \times \PP^1$, we have a $2 \times 2 \times 2$ tensor $W$ with six $2$-minors and a single $2 \times 2 \times 2$ hyperdeterminant. Its six $2$-minors 
are $F_{0\bullet \bullet}$,  $F_{1\bullet \bullet}$, $F_{\bullet 0 \bullet}$,  $F_{\bullet 1 \bullet}$, $F_{\bullet \bullet 0}$,  $F_{\bullet \bullet 1}$, and we denote the unique hyperdeterminant by $H$. When we say that a subset of the factors of the principal $A$-determinant $E_A$ vanish, we also mean that the rest of the factors do not vanish. 
\begin{proposition} \label{prop:factors 2x2x2}
The factors of the principal $A$-determinant $E_A$ for $\PP^1 \times \PP^1 \times \PP^1$ vanish  according to the following patterns:
\begin{itemize}
    \item[a)] Exactly one factor vanishes where there are $7$ choices. 
    \item[b)] Exactly two factors vanish where there are $\binom{7}{2}$ choices. 
    \item[c)] Exactly $F_{0\bullet \bullet}, F_{\bullet 0 \bullet}$, and $F_{\bullet \bullet 0}$ vanish; there are $8$ such choices corresponding to $8$ "corners" of the $2 \times 2 \times 2$ tensor $W$. 
    \item[d)] Exactly $F_{0\bullet \bullet}, F_{1 \bullet \bullet}, F_{\bullet 0 \bullet}, F_{\bullet 1 \bullet}$, and $H$ vanish; there are $3$ such choices corresponding to two pairs of opposite faces of $W$. 
    \item[e)] All $7$ factors vanish. 
\end{itemize}
\end{proposition}
\begin{proof}
This can be computed by saturating all possible combinatorial combinations of factors by $\langle \prod_{ijk} w_{ijk}\rangle$ and seeing which other factors are contained in the irreducible components.
 Alternatively, this is also a corollary of  Theorem \ref{thm:P1 x P1 x P1} which we will prove in Section \ref{sec:P1 x P1 x P1}. We  summarize the result in 
 Figure \ref{fig:222_factor_relations}. The top row corresponds to cases in a), and the second row corresponds to cases in b). In the third row only the right most configuration is realizable; this corresponds to cases in c). None of the types of configurations in the forth  and sixth rows can be realized. In the fifth row, only the left most configuration is possible; this corresponds to the cases in d). The unique possibility in the bottom row corresponds to the case in e). 
\end{proof}
\begin{remark}
    Figure \ref{fig:222_factor_relations}  
    captures all vanishing relations among the factors of the principal $A$-determinant for $\PP^1 \times \PP^1 \times \PP^1$.
A vanishing of a $2$-minor is represented by the corresponding face of the cube being shaded in brown, while the vanishing of $H$ is represented by the entire cube being shaded in blue. Each configuration represents multiple such configurations by symmetry. For instance, the first configuration in the top row stands for six such configurations, one for each face of the cube. The configurations are ordered so that every configuration on the same row consists of the same number of factors. We included both realizable and non-realizable configurations. The ones that are highlighted in green correspond to those that are realizable: the vanishing of such a group of factors will not cause any other factors to vanish. If a configuration is obtained by the addition of a single factor from another realizable configuration, a black arrow is included between the two configurations. A red arrow indicates an implication of vanishing of factors: if the vanishing of a set of factors implies the vanishing of further factors, we placed a red arrow between the corresponding configurations. 
\end{remark}
\begin{remark}
By Proposition \ref{prop:deficient-mldeg}, the complement of the real hypersurface arrangement in $(\RR^*)^8$ defined by $E_A(W) = 0$ consists of $W \in (\RR^*)^8$ where $\chi(Y_{W,1}) = 6$. Each region in the complement is determined by which side of the hypersurfaces 
defined by the six minors and the hyperdeterminant it lies on, positive or negative.
Hence, there are $2^7 = 128$ possible sign patterns in this case.
A simple computation using \verb|HypersurfaceRegions.jl|
shows that $68$ of these sign patterns are realizable. Out of these, $64 = \frac{128}{2}$ are all possible sign patterns corresponding to regions in $H^+$. The remaining four are 
\begin{align*}
    & + + + + + + - \\
    & + + - - - - - \\
    & - - + + - - - \\
    & - - - - + + - \\
\end{align*}
where signs correspond to 
$F_{0\bullet \bullet}$,  $F_{1\bullet \bullet}$, $F_{\bullet 0 \bullet}$,  $F_{\bullet 1 \bullet}$, $F_{\bullet \bullet 0}$,  $F_{\bullet \bullet 1}$, and $H$, in that order.
\end{remark}

\begin{figure}[H]
    \centering
    \includegraphics[scale=0.285]{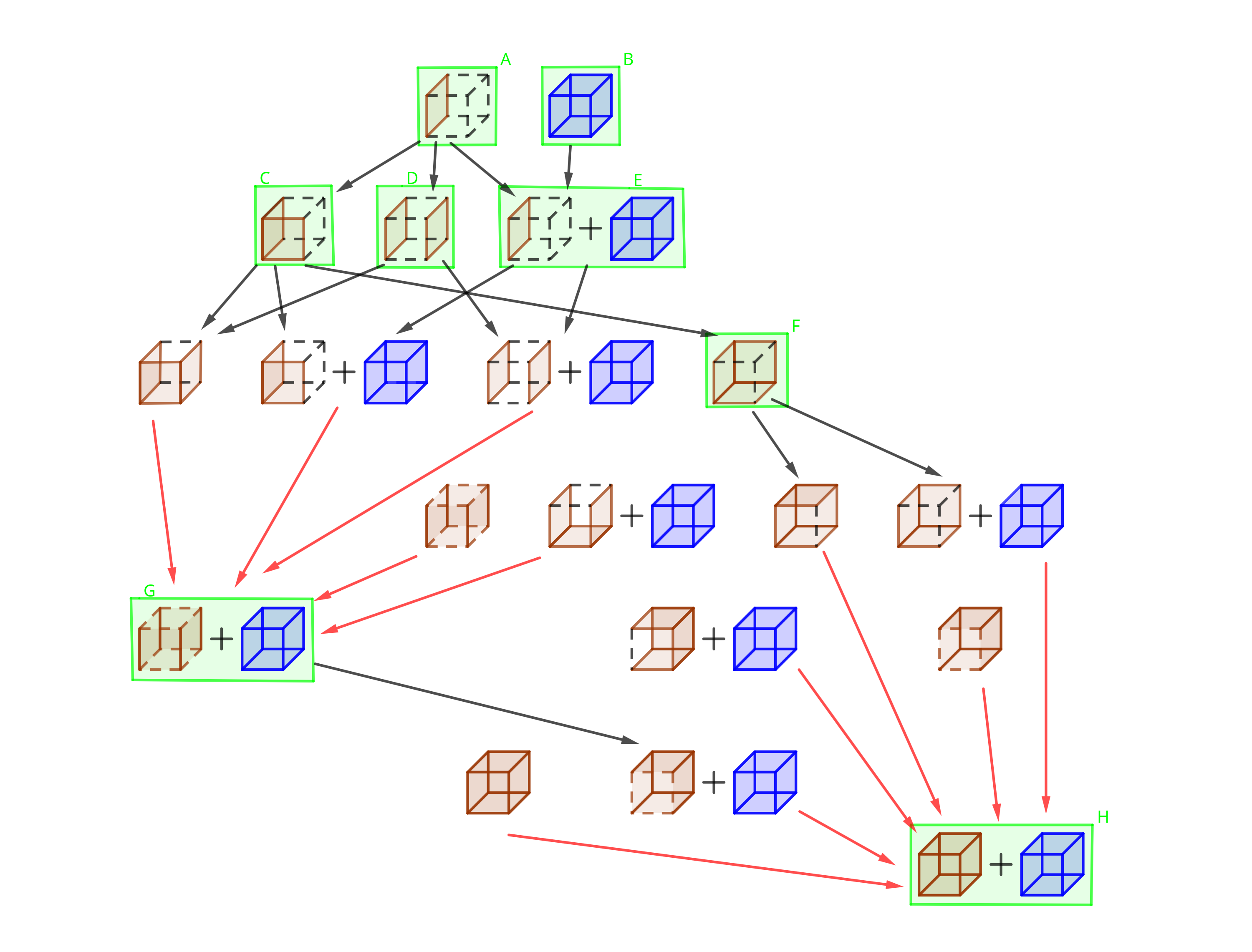}
    \caption{Vanishing relations among factors of $E_A$ for the case $n = 1$}
    \label{fig:222_factor_relations}
\end{figure}

\subsection*{$\PP^1 \times \PP^1 \times \PP^2$}
A complete classification of the vanishing patterns of the factors of $E_A$ in the case of $\PP^1 \times \PP^1 \times \PP^2$ is out of our reach. Here we briefly report the cases with two or three vanishing factors which give rise to further factors vanishing. Recall that now $W$ is a $2 \times 2 \times 3$ tensor. For ease of notation, we first introduce the following.

\begin{definition}
\begin{enumerate}
\item Let $F= \{F_1,F_2\}$ be a set of two minors not on the same face of $W$. We say $F$ is a \emph{hook} if the minors $F_1$ and $F_2$ share exactly two variables.
\item Let $F= \{F_1,F_2\}$ be as above. We say $F$ is a \emph{mirror} if
    \begin{itemize}
    \item there is a $2 \times 2 \times 2$ subtensor of $W$ containing $F$ and
    \item the variables in $F_1$ and $F_2$ are disjoint.
    \end{itemize}
\item Let $F= \{F_1,F_2,F_3\}$ be a set of three distinct minors of $W$. Then $F$ is a \emph{square cup} if
    \begin{itemize}
    \item there is a $2 \times 2 \times 2$ subtensor of $W$ containing $F$ and
    \item the variables in $F_1$ and $F_2$ are disjoint.
    \end{itemize} 
 \item Let $F= \{F_1,F_2,F_3,F_4\}$ be a set of distinct minors  of $W$. We say $F$ is a \emph{cubic frame} if
    \begin{itemize}
    \item there is a $2 \times 2 \times 2$ subtensor of $W$ containing $F$,
    \item the variables in $F_1$ and $F_2$ are disjoint and
    \item the variables in $F_3$ and $F_4$ are disjoint.
    \end{itemize} 
\end{enumerate}
\end{definition}
In the list below we describe (almost) all configurations with two or three vanishing factors which imply further factors vanishing. For this subsection only, we denote the unique $2 \times 2 \times 3$ hyperdeterminant by $H$. 
We studied all possible combinatorial cases as illustrated in the following example:

\begin{example}
    Assume the vanishing of a mirror with minors contained in two $2 \times 3$ faces of $W$, say $F_{0 \bullet (0,1)} = F_{1 \bullet (0,1)} =0$ as well as $H=0$. The saturation of this ideal by $\langle \prod_{ijk} w_{ijk} \rangle$ has four minimal primes. The first one contains $F_{0 \bullet (0,2)}$ and $F_{0 \bullet (1,2)}$, the second one contains $F_{1 \bullet (0,2)}$ and $F_{1 \bullet (1,2)}$, the third one contains $F_{\bullet 0 (0,1)}$, $F_{\bullet 1 (0,1)}$ and $H_{01}$ and the fourth one contains $F_{\bullet \bullet 0}$, $F_{\bullet \bullet 1}$ and $H_{01}$. Thus, either one of the sides $W_{0 \bullet \bullet}$ or $W_{1 \bullet \bullet}$ has all its minors vanish, or $H_{01}$ as well as one of the cubic frames containing the mirror $F_{0 \bullet (0,1)}$, $F_{1 \bullet (0,1)}$ vanishes. This is precisely case 5 in the list below.
\end{example}

\begin{enumerate}
\item Two minors within the same face of $W$ vanish $\Longrightarrow$ all three minors within that face and $H$ vanish. 

\item The minors in a square cup vanish 
$\Longrightarrow$ the minors of the cubic frame containing the square cup in the same $2 \times 2 \times 2$ subtensor, its hyperdeterminant, as well as $H$ vanish. 

\item The minors in a  hook and the  $2\times 2\times 2$-hyperdeterminant of the corresponding subtensor vanish 
$\Longrightarrow$ the minors of the cubic frame containing the hook and $H$ vanish as well.

\item The minors in a mirror and the  $2\times 2\times 2$-hyperdeterminant of the corresponding~subtensor vanish 
$\Longrightarrow$ the minors in one of the cubic frames containing the mirror and $H$ vanish as well.

\item A mirror with minors contained in two $2 \times 3$ faces $\{S_1,S_2\}$ of $W$ and $H$ vanish \\
\begin{minipage}[t]{0.01\textwidth}
$\Longrightarrow$ 
\end{minipage}
\begin{minipage}[t]{0.8\textwidth}
\begin{itemize}
\setlength{\itemsep}{0pt}
    \item all minors within $S_1$ or $S_2$ vanish or
    \item the minors in one of the cubic frames containing the mirror and the hyperdeterminant of the corresponding $2 \times 2 \times 2$ subtensor vanish. 
\end{itemize}
\end{minipage}

\item A hook with minors contained in two adjacent $2 \times 3$ faces  $\{S_1,S_2\}$ of $W$ and $H$ vanish 
\begin{minipage}[t]{0.01\textwidth}
$\Longrightarrow$ 
\end{minipage}
\begin{minipage}[t]{0.8\textwidth}
\begin{itemize}
\setlength{\itemsep}{0pt}
    \item all minors within $S_1$ or $S_2$ vanish or
    \item the minors in the cubic frame containing the hook and the hyperdeterminant of the corresponding the $2\times 2 \times 2$ subtensor  vanish. 
\end{itemize}
\end{minipage}

\item A $2\times 2 \times 2$-hyperdeterminant $G$, a minor $U$ coming from  a  slice of $W$  contained in the corresponding subtensor of $G$ and $H$ vanish \\
\begin{minipage}[t]{0.01\textwidth}
$\Longrightarrow$ 
\end{minipage}
\begin{minipage}[t]{0.8\textwidth}
\begin{itemize}
\setlength{\itemsep}{0pt}
    \item  the $2\times 2\times 2$-hyperdeterminant of the other subtensor containing $U$ vanish or
    \item the minors in one of the cubic frames containing $U$ vanish. 
\end{itemize}
\end{minipage}

\item A $2\times 2 \times 2$-hyperdeterminant $G$, a minor $U$ coming from  a $2 \times 3$ face $S$ of $W$ and contained in the corresponding subtensor of $G$ and $H$ vanish \\ 
 \begin{minipage}[t]{0.01\textwidth}
$\Longrightarrow$ 
\end{minipage}
\begin{minipage}[t]{0.8\textwidth}
\begin{itemize}
\setlength{\itemsep}{0pt}
    \item all minors within $S$ vanish or
    \item the minors in one of the cubic frames containing $U$ vanish. 
\end{itemize}
\end{minipage}

\item Two $2\times 2 \times 2$-hyperdeterminants and the minor fully contained in both  corresponding subtensors vanish
$\Longrightarrow $ $H$ vanishes as well.
\end{enumerate}

This list is complete except for two cases, for which the computations were too expensive. We do not know if there are any implications if
\begin{itemize}
\setlength{\itemsep}{0pt}
    \item all three $2 \times 2 \times 2$-hyperdeterminants vanish or
    \item two of the $2 \times 2 \times 2$-hyperdeterminants and $H$ vanish.
\end{itemize}

Additionally, the following result will be useful for proving Theorem \ref{thm:realizability}.
\begin{lemma}\label{lem: minors H}
    A set of minors vanishing will cause $H$ to vanish if and only if that set contains either a square cup or two out of three minors of a face of $W$.
\end{lemma}
\begin{proof}
    This can be seen by exhaustive computation.
\end{proof}

\section{ML degree of three-way independence models} \label{sec: three-way}

%\subsection{ML degree of Cayley polytopes}

In this section we start with observations regarding the ML degree of general scaled toric varieties. 
Let $f(x, z)$ be a Laurent polynomial of the form \begin{equation*} \label{eq:cayley_polynomial}
 f(x, z) = f_0(x) + \sum_{k=1}^n z_k f_k(x), \end{equation*}
 where $f_0, \dots, f_n \in \mathbb{C}[x_1^{\pm 1}, \dots, x_m^{\pm 1}]$ are Laurent polynomials in $m$ variables. Let $A \subseteq \ZZ^{n+m}$ denote the set of exponent vectors of the terms in $f(x, z)$, and let
 $W = ( w_a \in \CC^*  \suchthat a \in A)$  be the coefficients of these terms.  The scaled toric variety $X_{A,W} \subseteq \P^{\abs{A}-1}$ is given by the monomial parametrization  $(\C^*)^{n+m} \mapsto \P^{\abs{A}-1}$
\[ (t_1, \ldots, t_{n+m}) \mapsto [w_a t^a \suchthat a \in A] \in \P^{\abs{A}-1}. \]
As in the case of Proposition \ref{prop:ML degree = Euler char}, it is known from \cite[Theorem 1]{H13} and \cite[Section 6.4]{TW24} that the ML degree of  $X_{A,W}$ is equal to 
\[ \mldeg(X_{A,W}) = - (-1)^{n+m} \chi (V(f) \cap (\CC^*)^{n+m}), \]
where $V(f)$ is the affine hypersurface %in $(\CC^*)^{n+m}$ 
defined by the vanishing of $f(x,z)$. Our starting point is the following result.
\begin{theorem} \label{thm:Fevola+Matsubara-Heo}
\cite[Theorem 2.2]{FM24}
Let $f(x,z) = f_0(x) + \sum_{k=1}^n z_k f_k(x)$ be a Laurent polynomial where  $f_0,\ldots, f_n \in \CC[x_1^{\pm}, \ldots, x_m^{\pm}]$. Then
$$ \chi\left(V(f) \cap (\CC^*)^{m+n}\right) = (-1)^n \chi\left(V(f_0f_1\cdots f_n) \cap (\CC^*)^m\right).$$ 
\end{theorem}
%\begin{theorem} \label{thm:Fevola+Matsubara-Heo}
%\cite[Theorem 2.2]{FM24}
%For $f(x,z)$ as in \eqref{eq:cayley_polynomial} we have
%$$ \chi\left(V(f) \right) = (-1)^n \chi\left(V(f_0f_1\cdots f_n)\right).$$ 
%\end{theorem}
Throughout the rest of the paper, for any positive integer $k$, we will use the notation 
 $[\bar{k}]\coloneqq \{0, 1, \ldots, k\}$. Now, for $\varnothing \neq I \subseteq [\bar{n}]$ let us denote by $C_I$ the intersection in the torus
\[ C_I = V(f_k \suchthat k \in I) \subseteq (\CC^*)^m. \] 
Via elementary exclusion-inclusion, one can go from the union $V(f_0 \ldots f_n)$ to intersections and obtain the following statement. 
\begin{corollary}
\label{cor:euler_characteristic_via_intersections}
    The Euler characteristic of $V(f) \cap (\CC^*)^{m+n}$ is equal to 
    \[ \chi( V(f) \cap (\CC^*)^{m+n}) =  (-1)^{n+1} \sum_{\varnothing \neq I \subseteq [\bar{n}]} (-1)^{\abs{I}} \chi( C_I), \] 
    and therefore the ML degree of $X_{A,W}$ is equal to 
    \[ \mldeg(X_{A,W}) = (-1)^m \sum_{\varnothing \neq I \subseteq [\bar{n}]} (-1)^{\abs{I}} \chi( C_I).\]
\end{corollary}

Now  consider $f(x,z)$ such that the set of  its exponent vectors  is the vertex set of a product of $r+1$ simplices $\Delta_{d_1} \times \ldots \times \Delta_{d_r} \times \Delta_n$. The last $n$-dimensional simplex corresponds to the coordinates $z_1,\ldots, z_n$ in $f(x,z)$. 
%The equations $f_i(x)$ associated to this setup are multilinear of degree $r$.  We hope to exploit this structure and, 
Instead of considering the intersections $C_I$ in the torus $(\CC^*)^{d_1} \times \cdots \times (\CC^*)^{d_r}$, we would like to consider them in the product of projective spaces $\P^{d_1} \times \cdots \times \P^{d_r}$.  We denote the homogenous coordinates of the factor $\P^{d_i}$ by $x^{(i)}_0,\ldots,x^{(i)}_{d_i}$.  In these coordinates, the homogenized versions of $f_k(x)$ will be 
denoted by 
$$q_k(x_0^{(1)}, \ldots,x_{d_1}^{(1)}, \ldots,x_0^{(r)}, \ldots,x_{d_r}^{(r)} ). $$
Note that with these coordinates we can organize the coefficients of $f(x,z)$ into a tensor $W=(w_{i_1 \ldots i_r k})$ of format $ (d_1+1) \times \ldots \times (d_r+1)\times(n+1)$ where $w_{i_1 \ldots i_r k}$ is the coefficient of $x_{i_1}^{(1)} \cdots x_{i_r}^{(r)}  z_{k}$. 
Let $V_I$ denote now the intersection in $\P^{d_1} \times \ldots \times \P^{d_r}$ defined by 
\[ V_I = V( q_k \suchthat  k \in I). \] 
For a subset $ \mathcal{J} = J_1 \times \ldots  \times J_{r} \subseteq [\bar{d}_1] \times \ldots \times [\bar{d}_{r}]$ let us introduce the notation
\[ X_{\mathcal{J}} = \set{ x^{(i)}_j=0 \suchthat j \in J_i    } \subseteq \P^{d_1} \times \ldots \times \P^{d_{r}}. \] 
%Using the exclusion-inclusion once again one can easily reformulate Corollary \ref{cor: ML_via_intersections} in terms of $V_I$ and $X_\mathcal{J}$.
With this we get a formula for the ML degree of the scaled toric variety associated to a product of simplices as in Corollary \ref{cor:euler_characteristic_via_intersections}. 
\begin{theorem} \label{thm: ML_product}
Let $A$ be the vertex set of the product of simplices $\Delta_{d_1} \times \cdots \times \Delta_{d_r} \times \Delta_n$ and $W$ be a scaling considered as a tensor of format $(d_1+1, \ldots, d_r+1, n+1)$. Then 
\[ \mldeg(X_{A,W}) = (-1)^{\sum_{i=1}^{r} d_i} \sum_{ \varnothing \neq I \subseteq [\bar{n}]} (-1)^{\abs{I}} \left[  \sum_{\substack{\mathcal{J}: J_i \subsetneq  [\bar{d}_i]}} (-1)^{\abs{\mathcal{J}}}\chi \left( V_I \cap X_{\mathcal{J}} \right) \right],  \] 
where $\abs{\mathcal{J}} \coloneqq \abs{J_1} + \ldots + \abs{J_{r}}$. 
\end{theorem}
\begin{proof}
We start with the identity 
$C_I = V_I \cap \{ x_{j_i}^{(i)} \neq 0 \suchthat i=1,\ldots,r, \,\, j_i =0, \ldots, d_i \} $. The second term in the intersection is equal to 
$$ \PP^{d_1} \times \cdots \times \PP^{d_r} \, \backslash \, \bigcup_{i=1}^r \bigcup_{j=0}^{d_i} \{x_j^{(i)} = 0 \}. $$
Inclusion-exclusion combined with Corollary \ref{cor:euler_characteristic_via_intersections} gives the result.
\end{proof}
We note that with Theorem \ref{thm: ML_product} we managed to express the ML degree in terms $\chi(V_I \cap X_{\mathcal{J}})$ and these only depend on the subtensors of $W$ that have indices belonging to the set $ ([\bar{d}_1] \setminus J_1) \times \ldots \times ([\bar{d}_{r}] \setminus J_{r}) \times I$. This has a nice consequence for the product of two simplices; see also \cite{CHKO24}.

%Moreover, note that at no point so far we used that the scalings must have entries in $\mathbb{C}^*$, therefore, this can also be used to deduce the full Euler stratification. 

%When $r=1$ we can use Lemma \ref{lem: rank_nullity} to immediately obtain a very explicit expression for the ML degree.   Using the fact that $\sum_{k=0}^a (-1)^k \binom{a+1}{k} \cdot (a+1-k)$ is zero, and switching from summation over $J$ to its complement we reobtain the result of \cite{CHKO24} in a slightly new form. 
\begin{corollary} \label{cor: ML-product-of-two}
    The ML degree of a product of two simplices $\Delta_{m} \times \Delta_{n}$  with the scaling $W$ considered as an $(m+1)\times (n+1)$ matrix  is equal to 
\[  \mldeg( X_{\Delta_{m} \times \Delta_{n},W}) = \sum_{\substack{ \varnothing \neq  J \subseteq [\bar{m}]\\  \varnothing \neq  I \subseteq [\bar{n}]}} (-1)^{\abs{I} +\abs{J}} \mathrm{rank}(W_{IJ}),    \] 
where $W_{IJ}$ is the submatrix whose rows and columns are indexed by $I$ and $J$, respectively.  %Note that this alternating sum is exactly the beta invariant of the matroid formed by the extended matrix of $W$ considered in \cite{CHKO24}. 
\end{corollary}
\begin{proof}
 Theorem \ref{thm: ML_product}   gives us 
 \[ \mldeg(X_{\Delta_m \times \Delta_n,W}) = (-1)^m \sum_{ \varnothing \neq I \subseteq [\bar{n}]} (-1)^{\abs{I}} \left[  \sum_{\substack{J \subsetneq  [\bar{m}]}} (-1)^{\abs{J}}\chi \left( V_I \cap X_J\right) \right].  \] 
 The intersection $V_I \cap X_J$ is a linear subspace in $\PP^{m-|J|}$ defined by the vanishing of linear forms whose coefficients are the entries of $W_{IJ'}$, where $J'$ is the complement of $J$ in $[\bar{m}]$. The Euler characteristic of this linear subspace is precisely $m-|J|+1 - \mathrm{rank}(W_{IJ'})$. Now using the fact that
 $\sum_{k=0}^m (-1)^k \binom{m+1}{k} \cdot (m+1-k)$ is zero and swapping $J$ and $J'$
 we arrive at the formula. 
\end{proof}

\subsection{Scaled $\PP^1\times\PP^1\times\PP^n$}
Now we specialize back to our main focus of the paper, namely three-way independence models
based on $\Delta_1 \times \Delta_1 \times \Delta_n$.
 Our goal is to prove Theorem \ref{thm:main}, and in order to do so, we are going to understand how the terms in Theorem \ref{thm: ML_product} behave with respect to the pattern of vanishing  of the factors of the principal $A$-determinant $E_A$. Theorem \ref{thm:main} is the combination of Corollary \ref{cor: part 1} and Corollary \ref{cor: main23} below.
Along the way, we will give explicit expressions for many terms in the correspoonding formula of Theorem \ref{thm: ML_product}. 

 We go back to the notation from earlier sections where  
 $x_i^{(1)}=x_i$ and $x_i^{(2)}=y_i$, and let   
$$q_k(x_0,x_1,y_0,y_1) = w_{00k}x_0y_0 + w_{01k}x_0 y_1 + w_{10k}x_1y_0  + w_{11k} x_1y_1. $$
%("q" stands for quadratic, since our equations are now bilinear). 
 
% Let us introduce the following notation
% \[ V_I \coloneqq\set{ (x,y) \in \P^1 \times \P^1 \suchthat q_i=0 \quad \forall ~  i \in I }, \]
% \[ V_{x_i} \coloneqq \set{ (x,y) \in \P^1 \times \P^1 \suchthat x_i=0 }, \]
% \[ V_{y_j} \coloneqq \set{ (x,y) \in \P^1 \times \P^1 \suchthat y_j=0 }. \]
% Using inclusion-exclusion we can write the Euler characteristic of the hypersurface arrangement $H_W$ as 
% \begin{align*}  
%  &\chi(H_W) = \sum_{i=0}^1 \chi( V_{x_i})) +\sum_{j=0}^1 \chi( V_{y_j})  - \sum_{i,j=0}^1 \chi( V_{x_i}\cap V_{y_j})  + \\&+\sum_{\varnothing\neq I \subseteq [n]} (-1)^{\abs{I}}  \left(-\chi(V_I) + \sum_{i=0}^1 \chi(V_I \cap V_{x_i}) + \sum_{j=0}^1 \chi(V_I \cap V_{y_j}) - \sum_{i,j=0}^1 \chi( V_I \cap V_{x_i} \cap V_{y_j}) \right).   
%  \end{align*}
\noindent
The terms  in Theorem \ref{thm: ML_product} for our case  can be subdivided into two types:  
$$ \mathrm{(i)}  \,\,\chi(V_I \cap X_{\mathcal{J}}) \mbox{ with } 1\leq \abs{\mathcal{J}}\leq 2 \quad \mbox{and} \quad \mathrm{(ii)} \, \, \chi(V_I).$$
%\begin{enumerate}[(i)]
%    \item $\chi(V_I \cap X_{\mathcal{J}})$ with $1\leq \abs{\mathcal{J}}\leq 2$;
    %\item $\chi(V_I \cap V_{x_i} \cap V_{y_j})$;
%    \item $\chi(V_I)$.
%\end{enumerate}
\begin{lemma}\label{lem: easy Euler char}
The Euler characteristic $\chi(V_I \cap X_{\mathcal{J}})$ with $1\leq \abs{\mathcal{J}}\leq 2$ depends only on which $2$-minors 
of $W$ vanish.  
\end{lemma}
\begin{proof}
 If $\abs{\mathcal{J}}=2$, then $X_\mathcal{J}$ is just a point in $\P^1 \times \P^1$. 
 Without loss of generality, it is given by $x_0 = y_0 = 0$. Substituting this into $q_k$ for $k\in I$, we see that $V_I \cap X_{\mathcal{J}} =  \emptyset$ since all coefficients in $W$ are in $\CC^*$. This shows that the Euler characteristic is $0$.
 %For $I=\{k_1,\ldots,k_{\abs{I}}\}$ we apply Lemma \ref{lem: rank_nullity} to a hyperplane arrangement in dimension $0$ to get
%\[ \chi\left(V_I \cap \{x_{1-i}=0\} \cap \{ y_{1-j}=0\}\right)  = 1 - \rank  \begin{pmatrix}
%    w_{ij k_1} \\
%    \vdots \\
%    w_{ij k_{\abs{I}}}
%\end{pmatrix}. \] 
%In particular, if we assume that the scaling entries lie in $\CC^*$, this is always zero. 

\noindent
If $\abs{\mathcal{J}}=1$, then $X_{\mathcal{J}} = \mathrm{pt} \times \PP^1$ or $\PP^1 \times \mathrm{pt}$.
Therefore, we are looking at the intersection of $\abs{I}$ linear hyperplanes in $\P^1$. For example, for $\mathcal{J}=\{1-i\} \times \varnothing$, as in the proof of Corollary \ref{cor: ML-product-of-two} we get 
% Dealing with the first two types is rather straightforward.  Let us start with the case (i). %Set $\tilde{i}\coloneqq 1-i$ for $i \in \{0,1\}$,
% The variety $V_I \cap V_{x_{1-i}}$ is defined by the equations $x_{1-i}=0$ and 
% \[ x_{i} \left( w_{i0k} y_0 + w_{i1k} y_1 \right)=0 \quad \text{for }k\in I. \] 
% Thus, as discussed in the previous section, the Euler characteristic of this for $I=\{k_1,\ldots, k_{\abs{I}}\}$ \vadym{maybe take wlog $I=\{0,\ldots,s\}$ for clarity}  is 
\[ \chi(V_I \cap  \{ x_{1-i}=0\}) = 2 - \mathrm{rank} \begin{pmatrix}
    w_{i0 k_1} & w_{i1 k_1} \\ 
    \vdots & \vdots  \\
     w_{i0 k_{\abs{I}}} & w_{i1 k_{\abs{I}}}   \\ 
\end{pmatrix}.  \]
The matrix above is a $ \abs{I}\times 2$ matrix lying on the face $W_{i\bullet \bullet}$. 
Similarly for $V_I \cap \{ y_{1-j}=0\}$ one gets a $ \abs{I} \times 2 $ matrix on the face $W_{\bullet j \bullet}$. Since we assume that $W \in (\C^*)^{4n+4}$, the vanishing pattern of the $2$-minors completely governs the ranks of the above matrices.  %Indeed, for an $m\times2$ matrix with non-zero entries, all its $2$-minors vanish if and only if $m-1$ of its $2$-minors vanish.  Thus,  we only have to look at the minors $F_{i \bullet (\alpha,\beta)}$  and $F_{ \bullet j(\alpha,\beta)}$ for $\alpha,\beta \in I$.
\end{proof}

The remainder of this section is devoted to dealing with $\chi(V_I)$. 
With $I= \{k_1, \ldots, k_{\abs{I}} \}$,
%To simplify the notation, we focus on $V(q_0, \ldots, q_s)$ with $s=\abs{I}-1$; every other combination follows in a similar manner. 
the system 
$q_{k_1}=\ldots = q_{k_{|I|}} =0$
%\begin{align*} \label{matrix_system}
%    0 = q_{k_1} &= w_{00k_1}x_0y_0 + w_{01k_1}x_0y_1 + w_{10k_1}x_1 y_0 + w_{11k_1} x_1 y_1 \\
%    & \vdots \\
%    0 = q_{k_{\abs{I}}} &= w_{00 k_{\abs{I}}}x_0y_0 + w_{01k_{\abs{I}}}x_0y_1 + w_{10k_{\abs{I}}}x_1 y_0 + w_{11k_{\abs{I}}} x_1y_1 
%\end{align*}
can be written as 
\begin{align*}
0 = T_{I}(y) \cdot x := 
\begin{pmatrix}
    w_{00k_1}y_0 + w_{01k_1}y_1 & w_{10k_1}y_0 + w_{11k_1}y_1 \\
    \vdots & \vdots \\
    w_{00k_{\abs{I}}}y_0 + w_{01k_{\abs{I}}}y_1 & w_{10k_{\abs{I}}}y_0 + w_{11k_{\abs{I}}}y_1
\end{pmatrix}
\cdot \begin{pmatrix}
    x_0 \\ x_1
\end{pmatrix}.
\end{align*}
%We are going to abuse notation and use $T_I(y)=T_{k_1\ldots k_{\abs{I}}}(y)$ whenever suitable. 
Note that for a fixed $y \in \PP^1$, this system admits a solution $x \in \PP^1$ if and only if $\rank T_I(y) <2$. 

\subsubsection*{\underline{The cases with $1\leq \abs{I} \leq 2$}}

We start with the case $\abs{I}=1$.
\begin{lemma}\label{lem: one quadric}
For $W \in (\mathbb C^*)^{4n+4}$, the Euler characteristic $\chi(V_{\{k\}})$ depends only on whether $F_{\bullet \bullet k}$ vanishes. More precisely,
$ \chi (V_{\{k\}})  = 4 - \rank W_{\bullet \bullet k}$. 
\end{lemma}
\begin{proof}
In this situation the rank of $T_k(y)$ is always at most $1$, but there might exist a point $y$ where the rank drops to $0$. This happens if and only if $w_{00k} y_0 + w_{01k} y_1 = \lambda ( w_{10k} y_0 + w_{11k} y_1)$ for $\lambda \in \CC^*$. In turn, this happens if and only if $\det W_{\bullet \bullet k} = F_{\bullet \bullet k}=0$.  When $\rank W_{\bullet \bullet k} = 2$, $V_{\{k\}}$ is of the form
$\PP^1 \times \mathrm{pt}$, and when $\rank W_{\bullet \bullet k} = 1$, $V_{\{k\}}$ is 
of the form 
$((\PP^1 \setminus \mathrm{pt} )\times \mathrm{pt}) \sqcup (\mathrm{pt} \times \PP^1)$. In both cases, $\chi(V_{\{k\}}) = 4 - \rank W_{\bullet \bullet k}$. 
%Thus, since rank of $W_{\bullet \bullet k}$ is never $0$ and using Lemma \ref{lem: rank_nullity}, we can write  
%\[ \chi (V_{\{k\}})  = ( \chi (\mathbb P^1) - (2- \rank W_{\bullet \bullet k})) \cdot \chi( \pt) + (2- \rank W_{\bullet \bullet k}) \cdot \chi( \P^1) = 4 - \rank W_{\bullet \bullet k}.  \] 
\end{proof}

For the case  $\abs{I}=2$, we let $I = \{i,j\}$. We can still express the contribution $\chi(V_I)$ explicitly, but it is a little bit more involved. For a solution in $\PP^1 \times \PP^1$ to $T_{ij}(y) \cdot x =0 $ to exist, $\det T_{ij}(y)$ must vanish. 

\begin{lemma}
    The determinant of  $T_{ij}(y)$ is 
    \[     \det T_{ij}(y) = F_{\bullet 0 (i,j)} y_0^2 
    + (w_{00i}w_{11j} + w_{01i}w_{10j}-w_{10i}w_{01j}-w_{11i}w_{00j})y_0y_1 + F_{\bullet 1 (i,j)} y_1^2,  \]
and the discriminant of this quadric is exactly the hyperdeterminant $H_{ij}$.  
\end{lemma}

The set $\{ y \in \P^1 \suchthat \det T_{ij}(y) =0 \}$ can consist of  $1$ or $2$ points or be the whole $\P^1$.  However, to understand the geometry of $V_I$, it is not enough to just know the points where  $\det T_{ij}(y)=0$, we also need to know the rank of $T_{ij}(y)$ at those points. The matrix $T_{ij}(y)$ can have rank $0$ at some point $y$ only if it factorizes as 
 \begin{equation} \label{eq: Ty_factorization} T_{ij}(y) = (a_0 y_0 + a_1 y_1) \cdot M, \end{equation} 
with $a_i \in \mathbb{C}$ and $M$ a constant $2\times2$ matrix.  We note that in this case the corresponding hyperdeterminant $H_{ij}$ vanishes.

If $T_{ij}(y)$ does not factorize as above, then there are a few possible scenarios. 
If $H_{ij}$ does not vanish, then 
 $\det T_{ij}(y)=0$ at exactly two points. Otherwise, when $H_{ij}$ does vanish, we can have either one point at which $ \rank T_{ij}(y)=1$, since $H_{ij}$ is the discriminant of the quadratic equation $\det T_{ij}(y)=0$, or $\rank T_{ij}(y)=1$ for any point $y$, in which case all the coefficients of the equation $\det T_{ij}(y)=0$ are identically zero. 
 
 If the rank of $T_{ij}(y)$ is $1$ at some $y$, we find exactly one $x \in \P^1$ to satisfy $T_{ij}(y) \cdot x =0$. If the rank of $T_{ij}(y)$ is $0$ at some $y$, all $x \in \P^1$ satisfy $T_{ij}(y) \cdot x =0$. 
 We summarize this discussion as follows.

%\textcolor{purple}{Elke: I understand your idea with putting all of the general lemmas into the proof to have the easy case first and then refer back to it, but it might be hard to understand where the proof ends? We could do like: We summarize a general version of this discussion in the following lemma. It, together with Lemma \ref{lem: type characterization} finalizes the proof. (and then end the proof) ? (ok with the statement, the proof needs a bit more probably...)}
\begin{lemma} \label{lem: types_of_Ty}
The matrix $T_{ij}(y)$ falls into one of the following five types.
\begin{enumerate}[I:]
    \item \textbf{$T_{ij}(y)$ has rank 1 at two points and never rank 0}. \\
    This happens if the hyperdeterminant $H_{ij}$ does not vanish.
    Then $T_{ij}(y) \cdot x =0$ has two solutions $\pt \times \pt \,\sqcup \, \pt \times \pt$ and $\chi(V_I)=2$.
    
    \item \textbf{$T_{ij}(y)$ has rank 0 at one point only and rank $< 2$ nowhere else}. \\
    This happens if $T_{ij}(y) = (a_0 y_0+ a_1 y_1)M$ with $\rank M = 2$.
    Then $T_{ij}(y) \cdot x =0$ has as solution $\P^1 \times \pt$ and $\chi(V_I)=2$.
    \item \textbf{$T_{ij}(y)$ has rank 0 at one point only and rank 1 everywhere else}. \\
    This happens if $T_{ij}(y) = (a_0 y_0+a_1y_1)M$ with $\rank M =1$. 
    Then the set of solutions to $T_{ij}(y) \cdot x=0$ has the form 
    $(\PP^1 \times \pt) \sqcup \pt \times (\P^1 \setminus \pt)$ and $\chi(V_I)=3$.
    \item \textbf{$T_{ij}(y)$ has always rank 1 and never rank 0}. \\
    This happens if $\det T_{ij}(y)$ is identically zero. Then the solution set to $T_{ij}(y) \cdot x =0$ is isomorphic to $\P^1$ and $\chi(V_I)=2$. %There are two subcases:
    %\begin{itemize}
     %   \item The rows of $(T_y)_{ij}$ are proportional by a number and the columns are proportional by a rational function. This means that none of the rows can have a factor, since otherwise the other row would also have that factor, which would lead to case (3). 
     %   \item The columns of $(T_y)_{ij}$ are proportional by a number and the columns are proportional by a rational function. This means that both rows have a (distinct) factor.
    %\end{itemize}
    \item \textbf{$T_{ij}(y)$ has rank 1 at one point only}. \\
    This happens if $\det T_{ij}(y)$ is not identically zero but the hyperdeterminant $H_{ij}$ vanishes. Then $T_{ij}(y) \cdot x =0$ has one solution $\pt \times \pt$ and $\chi(V_I)=1$.
\end{enumerate}
\end{lemma}  
We want to understand what kind of coefficient tensor $W$ corresponds to each of the above types. 
For a $3$-way tensor $W = (w_{ijk}) \in \CC^p \times \CC^q \times \CC^r$ and $s=1,2,3$,
we let $W^{(s)}$ be the flattening matrix along mode $s$. For instance, $W^{(1)}$ is the $p \times (qr)$ matrix 
where the entry in row $i$
and column indexed by $(j,k)$ is  $w_{ijk}$. In particular for the $2\times2\times2$ tensor $ W_{\bullet \bullet(i,j)}$ we denote the corresponding flattenings by $W^{(s)}_{ij}$.
%and we define  $r_{s,ij} \coloneqq \rank W^{(s)}_{ij}$. 

%\begin{definition}
%    Let  $W=(w_{ijk})$ be a tensor. A mode-$m$ fibre of $W$ is a vector obtained by fixing all indices of $W$ except the $m$-th one. The mode-$m$ flattening $L^{(m)}$ of $W$ is the matrix whose columns are the $m$-modes of $W$.  In particular for the $2\times2\times2$ tensor $ W_{\bullet \bullet(i,j)}$ we denote the corresponding flattening by $L^{(m)}_{ij}$ and we define
%     \[ r_{m,ij} \coloneqq \rank L^{(m)}_{ij}. \] 
%     \[ L^{(m)} = \left(\begin{array}{c|c}
%      w_{0 \bullet \bullet} & w_{1 \bullet \bullet}
% \end{array} \right)  \] 
%\end{definition}

\begin{lemma}\label{lem: type characterization}
    Let $W \in (\C^*)^{4n+4}$. The type of the matrix $T_{ij}(y)$ depends only on the vanishing of the $2$-minors of $W_{\bullet \bullet (i,j)}$ and its $2 \times 2 \times 2$ hyperdeterminant $H_{ij}$.  The matrix $T_{ij}(y)$ is of type 
    \begin{enumerate}[I:]
    \item if $H_{ij} \neq 0$,
    \item if $H_{ij}=0$ and $F_{\bullet \bullet i}=F_{\bullet \bullet j}=F_{ 0 \bullet (i,j)}=F_{ 1 \bullet  (i,j)}=0$ and 
    $F_{\bullet 0 (i,j)} \neq 0$, $F_{\bullet 1 (i,j)} \neq 0$,
    \item if $H_{ij} = 0$ and all $2$-minors of $W_{\bullet \bullet (i,j)}$ vanish,
    \item if 
    $H_{ij} = 0$ and \\
    \qquad either  $F_{\bullet \bullet i} = F_{\bullet \bullet j} = F_{\bullet 0 (i,j)} = F_{\bullet 1 (i,j)} =0$ and 
    $F_{0 \bullet (i,j)} \neq 0$, $F_{1 \bullet (i,j)} \neq 0$ \\
    or $F_{0 \bullet (i,j)} = F_{1 \bullet (i,j)} = F_{\bullet 0 (i,j)} = F_{\bullet 1 (i,j)} =0$ and 
    $F_{\bullet \bullet i} \neq 0$, $F_{\bullet \bullet j} \neq 0$,
    \item if $H_{ij} =0$ and only one of the $2$-minors of $W_{\bullet \bullet (i,j)}$ vanishes.
\end{enumerate}
    
    \begin{proof}
    As already mentioned, if the hyperdeterminant $H_{ij} \neq 0$  then
    $T_{ij}(y)$ is of type I.  Assume from now on that $H_{ij} = 0$.

The matrix $T_{ij}(y)$ factors if and only if  the tuple $(w_{00i} ~ w_{00j} ~ w_{10i} ~ w_{10j})$ is proportional to the tuple $(w_{01i} ~ w_{01j} ~ w_{11i} ~ w_{11j})$ or in other words the matrix 
\[ W^{(2)}_{ij}=\begin{pmatrix}
    w_{00i} & w_{00j} & w_{10i} & w_{10j} \\
    w_{01i} & w_{01j} & w_{11i} & w_{11j} \\
\end{pmatrix} \]
has rank $1$. 
This happens if and only if $F_{\bullet \bullet i}=F_{\bullet \bullet j}=F_{ 0 \bullet (i,j)}=F_{ 1 \bullet  (i,j)}=0$. Hence we have a clear way to distinguish types II and III from the types IV and V. 

To decide whether $T_{ij}(y)$ falls into type II or III, one needs to compute the rank of the factor $M$ from \eqref{eq: Ty_factorization}. 
We see that $\rank M = 1$ if and only if 
        $F_{ \bullet 0 (i,j)} = F_{\bullet 1 (i,j)} = 0$, 
since
        $$ \rank M = \rank \begin{pmatrix}
            w_{00 i} & w_{10i} \\
            w_{00j} & w_{10j}
        \end{pmatrix}
        = \rank \begin{pmatrix}
            w_{01 i} & w_{11i} \\
            w_{01j} & w_{11j}
        \end{pmatrix}.
        $$
        %Hence, $\rank M = 1$ if and only if 
        %$F_{ \bullet 0 (i,j)} = F_{\bullet 1 (i,j)} = 0.$ 
        %Note that based on our assumption of a common factor, the matrices $W_{ \bullet 0 (i,j)}$ and $W_{\bullet 1 (i,j)}$ are proportional and thus drop in rank simultaneously. Moreover, note that the rank of $M$ coincides under the assumptions with the rank of another flattening, namely $\rank M =r_{3,ij}$. 
        
        Now assume that the factorization \eqref{eq: Ty_factorization} does not happen. 
        The rank of $T_{ij}(y)$ can be $1$ for any $y$ (in which case we obtain type IV) if either the rows or columns of $T_{ij}(y)$ are proportional. Thus either
        \begin{align*}
         \rank W^{(1)}_{ij} &= \rank \begin{pmatrix}
            w_{00i} & w_{01i} & w_{00j} & w_{01j} \\
            w_{10i} & w_{11i} & w_{10j} & w_{11j}
        \end{pmatrix} =1 \,\,\text{ or } \\
        \rank W^{(3)}_{ij} &= 
        \rank \begin{pmatrix}
            w_{00i} & w_{01i} & w_{10i} & w_{11i} \\
            w_{00j} & w_{01j} & w_{10j} & w_{11j}
        \end{pmatrix} =1.
        \end{align*}
        The first matrix has rank $1$ if and only if 
        $ F_{ \bullet \bullet i} = F_{\bullet \bullet j} = F_{\bullet 0 (i,j)} = F_{\bullet 1 (i,j)} = 0.$
        The second matrix has rank $1$ if and only if
        $ F_{ 0 \bullet (i,j)} = F_{1 \bullet (i,j)} = F_{\bullet 0 (i,j)} = F_{\bullet 1 (i,j)} = 0.$
        Otherwise, the matrix $T_{ij}(y)$ must be of type V.
    \end{proof}
\end{lemma}

Since for $n=1$ the size of $I$ is at most $2$, we have covered everything  needed  to compute the Euler characteristic $\chi(Y_{W,1})$ 
\begin{corollary}[Theorem \ref{thm:main}, part (1)] \label{cor: part 1}
    If exactly the same factors of the principal $A$-determinant $E_A$ vanish on $W$ and $W'$, then  $\chi(Y_{W,1})=\chi(Y_{W',1})$.
\end{corollary}

\subsubsection*{\underline{The cases with $\abs{I} \geq 3$}}

Now we can deal with the cases $\abs{I} \geq 3$.  If $\abs{I}=3$ and the corresponding $2\times2\times3$ hyperdeterminant does not vanish, then the resultant $\mathrm{Res}(q_k \, : \, k \in I)$ is nonzero, and hence $V_I = \emptyset$. 
 More generally, we can say the following. 
\begin{lemma}\label{lem: no solution}
	For an index set $I \subseteq [\bar{n}]$ with $|I|\geq 3$, 
	if one of the $2 \times 2 \times 3$ hyperdeterminants $H_{ijk}$ with $i,j,k \in I$ does not vanish, then $V_I =\varnothing$ and $\chi(V_I)=0$.  For $\abs{I}$=3, $V_I \neq \emptyset$ if and only if $H_{ijk}=0$.  
\end{lemma}
However, for $\abs{I} \geq 4$,  $V_I \neq \emptyset$ is not equivalent to the vanishing of all $2 \times 2 \times 3$ hyperdeterminants 
$H_{ijk}$ with $i,j,k \in I$.  We do not have anything further to report in this case, so we will treat only the case of $\abs{I}=3$ in the remainder. 

For $I=\{k_1,k_2,k_3\}$ and  $H_{k_1 k_2 k_3}=0$, 
in order to understand the geometry of $V_I$, we exploit the fact that 
\[ V_{\{k_1,k_2,k_3\}} = V_{\{k_1,k_2\}} \cap V_{\{k_1,k_3\}} \cap V_{\{k_2,k_3\}}, \] 
and our results about $V_{\{i,j\}}$. 
A special role is played by the flattening
$$W^{(3)}_I = 
\begin{pmatrix}
        w_{00k_1} & w_{01k_1} & w_{10k_1}  & w_{11k_1} \\
        w_{00k_2} & w_{01k_2} & w_{10k_2}  & w_{11k_2} \\
        w_{00k_3} & w_{01k_3} & w_{10k_3}  & w_{11k_3} \\
    \end{pmatrix}
    $$
of the $2\times2\times 3$ subtensor $W_{\bullet \bullet (k_1,k_2,k_3)}$.

\begin{remark}
Note that the rank of the above matrix $W_I^{(3)}$ is exactly what defines the singular locus of the $2\times 2 \times 3$ hyperdeterminant $H_{k_1 k_2 k_3}$ (see Proposition 5.4.(c) in \cite{WeymanZelevinsky1996}). Namely, the singular locus is exactly the set of $W$ for which $\rank W_I^{(3)} <3$. This locus cannot be described by the vanishing of certain factors of the principal $A$-determinant (see Example \ref{ex: counterexample} later).

For the $2 \times 2 \times 2$-hyperdeterminants $H_{ij}$ the singular locus is given by the set of $W$ where the flattenings $W_{ij}^{(2)}$ have rank $\leq 1$ (see the end of Section 2 in loc.cit.) As we already noted, if we consider $W$ only with entries in $\C^*$, then this is governed by the vanishing of the $2$-minors of the corresponding $2\times2\times2$-tensor. 
\end{remark}

\begin{proposition}\label{prop: chi of qs}
    Let $W \in (\C^*)^{4n+4}$ and consider $I = \{k_1,k_2,k_3\}$. Then, the Euler characteristic $\chi(V_I)$ depends only on the vanishing of the factors of $E_A$ if $H_{k_1 k_2 k_3}$ does not vanish at $W$ or if one of $H_{ij}=0$ with $i,j \in I$. Otherwise, 
    $\chi(V_I)$ depends on the factors of $E_A$ that vanish at $W$, as well as whether $\rank W^{(3)}_I =3$.  \end{proposition}

The above proposition and Lemma \ref{lem: no solution} together with the results for $\abs{I}\leq2$ combine to give the remaining parts of Theorem \ref{thm:main}. 
\begin{corollary}[Theorem \ref{thm:main}, parts (2) and (3)] \label{cor: main23}
If exactly the same factors of $E_A$ vanish on $W$ and $W'$, then $\chi(Y_{W,n}) = \chi(Y_{W',n})$ as long as 
\begin{itemize}
    \item $n\geq 2$  and at least one of the $2\times 2 \times 3$ hyperdeterminants $H_{ijk}\neq 0$ with $\{i,j,k\} \subseteq [\bar{n}]$
    \item $n=2$ and at least one of $H_{01},H_{02},H_{12}$ vanishes. 
\end{itemize}
\end{corollary}
It is instructive to write down the implications for $\PP^1 \times \PP^1 \times \PP^2$. 
\begin{corollary}\label{cor: chi depends on B}
    Let $n=2$ and $W \in (\C^*)^{12}$.
 Then $\chi(Y_{W,2})$ depends only on the factors of $E_A$ that vanish at $W$ and on whether the flattening $W^{(3)}$ has full rank. 
\end{corollary}

\begin{proof}[Proof of Proposition \ref{prop: chi of qs}]
Without loss of generality we assume that $I=\{0,1,2\}$. If $H_{012}\neq 0$, then $\chi(V_I)=0$. In the remainder we assume that 
$H_{012} = 0$ and $V_I$ is not empty. We focus on the matrices $T_{ij}(y)$ for $i,j \in I$.  We use Lemma \ref{lem: types_of_Ty} extensively.

If at least one matrix $T_{ij}(y)$  is of type V, then it follows that $V_{ij}$, and thus $V_I$, consists of exactly one point with $\chi(V_I)=1$.   From now on assume that no $T_{ij}(y)$ is of type V.

Now suppose that one of the matrices, say $T_{01}(y)$, is of type II. Therefore there exists exactly one $y \in \P^1$ such that $\rank T_{01}(y)<2$. In order for the system $T_I(y) \cdot x =0$ to have a solution (as we assume it does), the other two matrices must also drop rank at $y$. 
If the matrices $T_{02}(y)$ and $T_{12}(y)$ are both of types II or III, then $T_I(y) \cdot x =0$ has solution of the form $\PP^1 \times \pt$ and $\chi(V_I)=2$. If at least one of $T_{02}(y)$ and $T_{12}(y)$ is of type I or IV, then $T_I(y) \cdot x =0$ has a solution of the form $\pt \times \pt$ and $\chi(V_I)=1$.  Now we assume that no  $T_{ij}(y)$ is of type II.

If all matrices $T_{ij}(y)$ are of type III, then their factorizations must coincide and thus $\chi(V_I)=3$. If at least one is of type IV and the rest is of type III, then there is at least one $T_{ij}(y)$ that has rank $1$ at each point $y \in \P^1$. Since the other matrices also have either rank $1$ or $0$ at each $y$, $T_I(y)$ has rank $1$ for each $y$, therefore, for each $y$ there is exactly one solution of $T_I(y) \cdot x=0$ and $\chi(V_I)=2$.  From now on assume that at least one $T_{ij}(y)$ is of type I. 

Note that it is not possible to have exactly two matrices $T_{ij}(y)$ of type III since this would force the third one to be of this type also.  Suppose we have only one matrix of type I, say $T_{01}(y)$. Then the types of our three matrices are either I,III,IV or I, IV,IV. Let $y^1,y^2$ be the two points giving $\rank T_{01}(y^i)=1$. In both possible cases, we have $\rank T_I(y^i)=1$, 
and thus $V_I$ is of the form $\pt \times \pt \sqcup \pt \times \pt$ and $\chi(V_I)=2$. 
%We may assume now that we have at least two of the matrices $T_{ij}(y)$ of type (I).

Now suppose that $T_{01}(y)$ and $T_{02}(y)$ are of type I. Then the quadratic equations $\det T_{01}(y) =0$ and $\det T_{02}(y) =0$ have exactly two solutions each and we can write 
\begin{align*}
\det T_{01}(y)&=(a_{01} y_0 + b_{01} y_1) (c_{01} y_0 + d_{01} y_1), \\
\det T_{02}(y)&=(a_{02} y_0 + b_{02} y_1) (c_{02} y_0 + d_{02} y_1)
\end{align*}
with $(a_{0i},b_{0i}) \neq (c_{0i},d_{0i})$. Since we assume that $V_I$ is nonempty, one or two factors between $\det T_{01}(y)$ and $\det T_{02}(y)$ must coincide.  

If $T_{12}(y)$ is of type IV with proportional rows or of type III, which also has proportional rows, then there exists $\mu \in \C^*$ such that $\det T_{01}(y) = \mu \det T_{02}(y)$. Thus, both factors coincide up to multiplicity and $\chi(V_I)=2$. 

If $T_{12}(y)$ is of type IV with columns proportional by a number $\mu$, then we have 
 \begin{align*}
    \det T_{01}(y) & =  (w_{001} y_0 + w_{011} y_1) ((\mu w_{000} - w_{100})y_0 + (\mu w_{010} - w_{110}) y_1) \\
    \det T_{02}(y) & =  (w_{002} y_0 + w_{012} y_1) ((\mu w_{000} - w_{100})y_0 + (\mu w_{010} - w_{110}) y_1).
    \end{align*}
    One factor is the same, but the other is not since we assume that the rows of $T_{12}(y)$ are not proportional. Hence, in this case,
    $V_I$ is a single point  and $\chi(V_I)=1$.
    Note here that a matrix of type IV has either its rows or its columns   proportional, but never both. Which one it is depends on the rank of $W_{ij}^{(2)}$ and therefore also on the vanishing pattern of the $2$-minors.
    
% If the rows of $T_{12}(y)$ are proportional by a number $\mu$, then 
%     $\det T_{01}(y) = \mu \det T_{02}(y)$ and both solutions agree. Hence, there are two solutions for $q_0=q_1=q_2=0$. Note that for a type (IV) submatrix, either its rows or its columns are proportional by a number, but not both. 
%     The columns are proportional by a number if and only if both rows have a (distinct) factor.
%     The rows are proportional by a number if none of the rows has a factor. Therefore, the rows of 
%     $T_{12}(y)$ are proportional if $T_{12}(y)$ is of type (III) or if it is of type (IV) with none of the rows having a factor. If $T_{12}(y)$ is of type (IV) with the columns being proportional by a number $\mu$, then
%     \begin{align*}
%     \det T_{01}(y) & =  (w_{001} y_0 + w_{011} y_1) ((\mu w_{000} - w_{100})y_0 + (\mu w_{010} - w_{110}) y_1) \\
%     \det T_{02}(y) & =  (w_{002} y_0 + w_{012} y_1) ((\mu w_{000} - w_{100})y_0 + (\mu w_{010} - w_{110}) y_1).
%     \end{align*}

We are only left to consider the case when $T_{12}(y)$ is also of type I. 
    The two factors of all three determinants agree if and only if 
    $$\tilde{W} := \begin{pmatrix}
           F_{\bullet 0 (0,1)}  & w_{000}w_{111} + w_{010}w_{101}-w_{100}w_{011}-w_{110}w_{001} & F_{\bullet 1 (0,1)}  \\
           F_{\bullet 0 (0,2)}  & w_{000}w_{112} + w_{010}w_{102}-w_{100}w_{012}-w_{110}w_{002} & F_{\bullet 1 (0,2)}  \\
           F_{\bullet 0 (1,2)}  & w_{001}w_{112} + w_{011}w_{102}-w_{101}w_{012}-w_{111}w_{002} & F_{\bullet 1 (1,2)}  \\
            \end{pmatrix}$$ 
    has rank one. 
    One can compute that
    $$ V( 2 \text{-minors of } \tilde{W} )  = V( 3 \text{-minors of } W_I^{(3)}).$$
    Hence,  $\chi(V_I)=2$ if and only if $\rank W_I^{(3)} < 3$ and $\chi(V_I)=1$ otherwise.
\end{proof}

The last part of the above proof suggests that $\chi(Y_{W,2})$ does not depend solely on the vanishing pattern of the factors of $E_A$. Indeed, below we present a counterexample to Conjecture \ref{conj:main-conjecture}. This also shows that in general, the ML degree of independence models does not depend on the vanishing pattern of the factors of the principal $A$-determinant.

\begin{example}\label{ex: counterexample}
Consider two scaling tensors $W,W' \in (\mathbb C^{*})^{2 \times 2 \times 3}$ with 
\begin{align*}
q_0 &= x_0 y_0 + 3 x_0 y_1 + 2 x_1 y_0 + 4 x_1 y_1,\\
q_1&= 2 x_0 y_0 +  x_0 y_1 + 4 x_1 y_0 + 6 x_1 y_1,\\
q_2&= 3 x_0 y_0 +  4x_0 y_1 + 6 x_1 y_0 + 10 x_1 y_1,\\
q_2' &= 3 x_0 y_0 +  3x_0 y_1 + 6 x_1 y_0 +  x_1 y_1,
\end{align*}
where  $q_0,q_1,q_2$ are the quadrics associated to $W$ and   $q_0,q_1,q_2'$ are the quadrics associated to $W'$. For both we have 
$H_{012} = F_{0 \bullet (0,1)} = F_{0 \bullet (0,2)} = F_{0 \bullet (1,2)} = 0$,
but no other factor of $E_A$ vanishes on $W$ or $W'$. However, one can check that 
$\rank W_{\{0,1,2\}}^{(3)} =2$, while $\rank {W'}_{\{0,1,2\}}^{(3)} =3$.

Let us take a closer look at the corresponding systems $T(y) \cdot x =0$. Denote 
$$ T(y) = \begin{pmatrix}
y_0 + 3 y_1 & 2 y_2 + 4 y_1 \\
2 y_0 + y_1 & 4 y_0+6 y_1 \\
3 y_0 + 4 y_1 & 6 y_0 + 10 y_1 
\end{pmatrix}, \quad
T'(y) = \begin{pmatrix}
y_0 + 3 y_1 & 2 y_0 + 4 y_1 \\
2 y_0 + y_1 & 4 y_0+6 y_1 \\
3 y_0 + 3 y_1 & 6 y_0 +  y_1 
\end{pmatrix}
$$
the matrices corresponding to $W$ and $W'$ respectively. Then
\begin{align*}
\det T_{01}(y) = \det T'_{01}(y) = \det T_{02}(y) & = 2 y_1(4 y_0 + 7 y_1), \\
\det T_{02}(y) &= -2 y_1 (4y_0 + 7 y_1), \\
\det T'_{02}(y) &= y_1 (y_0 - 9 y_1), \\
\det T'_{12}(y) &= -y_1 (22y_0 +17 y_1)
\end{align*}
Indeed,
$\left( T \left(\begin{smallmatrix}
    7 \\
    -4
\end{smallmatrix} \right) \right)
\cdot \left( \begin{smallmatrix}
    2 \\
    -5
\end{smallmatrix} \right) = 
\left( T  \left(\begin{smallmatrix}
    1 \\
    0
\end{smallmatrix} \right) \right)
\cdot \left( \begin{smallmatrix}
    2 \\
    -1
\end{smallmatrix} \right) =
\left( T'  \left(\begin{smallmatrix}
    1 \\
    0
\end{smallmatrix} \right) \right)
\cdot \left( \begin{smallmatrix}
    2 \\
    -1
\end{smallmatrix} \right) =
0 $.
The points $(\left(\begin{smallmatrix}
    2 \\
    -5
\end{smallmatrix} \right),
\left(\begin{smallmatrix}
    7 \\
    -4
\end{smallmatrix} \right)) $ and $
( \left(\begin{smallmatrix}
    2 \\
    -1
\end{smallmatrix} \right),
\left(\begin{smallmatrix}
    1 \\
    0
\end{smallmatrix} \right))
$
are solutions to $q_0=q_1=q_2=0$ and 
$( \left(\begin{smallmatrix}
    2 \\
    -1
\end{smallmatrix} \right),
\left(\begin{smallmatrix}
    1 \\
    0
\end{smallmatrix} \right))$
is the single solution to $q_0=q_1=q_2'=0$.
Hence, $\chi(V(q_0,q_1,q_2))=2$ and $\chi(V(q_0,q_1,q_2'))=1$.
% We knew both would have at least one solution, since the $2 \times 2 \times 3$ hyperdeterminant vanishes. The fact that $q_0=q_1=q_2$ has two solutions comes from the singularity of $B_W$.
Overall, we obtain
$\chi(Y_{W,2})=-8$ and $\chi(Y_{W',2})=-9$ and thus $\mldeg(X_{W,n})=8$ and $\mldeg(X_{W',n})=9$.

%with mode-$3$ flattenings
%    \[ B_{W,\{0,1,2\}} = %\begin{pmatrix}
%1 & 3 & 2 & 4 \\
%2 & 1 & 4 & 6 \\
%3 & 4 & 6 & 10 \\
%\end{pmatrix}, \quad
%B_{W',\{0,1,2\}} = %\begin{pmatrix}
%1 & 3 & 2 & 4 \\
%2 & 1 & 4 & 6 \\
%3 & 3 & 6 & 1 \\
%\end{pmatrix}\]
%For both we have 
%$$H_{012} = F_{0 \bullet (0,1)} = F_{0 \bullet (0,2)} = F_{0 \bullet (1,2)} = 0$$
%but none of the other factors vanish. 
%Denote 
%\begin{align*}
%q_0 &= x_0 y_0 + 3 x_0 y_1 + 2 x_1 y_0 + 4 x_1 y_1,\\
%q_1&= 2 x_0 y_0 +  x_0 y_1 + 4 x_1 y_0 + 6 x_1 y_1,\\
%q_2&= 3 x_0 y_0 +  4x_0 y_1 + 6 x_1 y_0 + 10 x_1 y_1,\\
%q_2' &= 3 x_0 y_0 +  3x_0 y_1 + 6 x_1 y_0 +  x_1 y_1.\\
%\end{align*}
%The quadrics in $\mathbb P^1 \times \mathbb P^1$ associated to $W$ are $q_0,q_1,q_2$. The quadrics associated to $W'$ are $q_0,q_1,q_2'$. 
If we call the respective quadrics in $\CC^1 \times \CC^1$ as $f_0, f_1, f_2, f_2'$,
one can observe that $f_0=f_1=f_2=0$ has two solutions, one of them on the $x$-axis. In contrast, $f_0=f_1=f_2'=0$ has only one solution, which also lies on the $x$-axis. This comes from the fact that $f_2$ lies in the pencil of $f_0$ and $f_1$ while $f_2'$ does not, which again is directly linked to the ranks of the flattenings.
%\begin{align*}
%f_0 &= 1 + 3  y_1 + 2 x_1  + 4 x_1 y_1,\\
%f_1 &= 2  +  y_1 + 4 x_1 + 6 x_1 y_1,\\
%f_2 &= 3+  4 y_1 + 6 x_1 + 10 x_1 y_1,\\
%f_2' &= 3 +  3y_1 + 6 x_1 +  x_1 y_1.\\
%\end{align*}
\begin{figure}[H]
        \centering
        \includegraphics[width=0.45\linewidth]{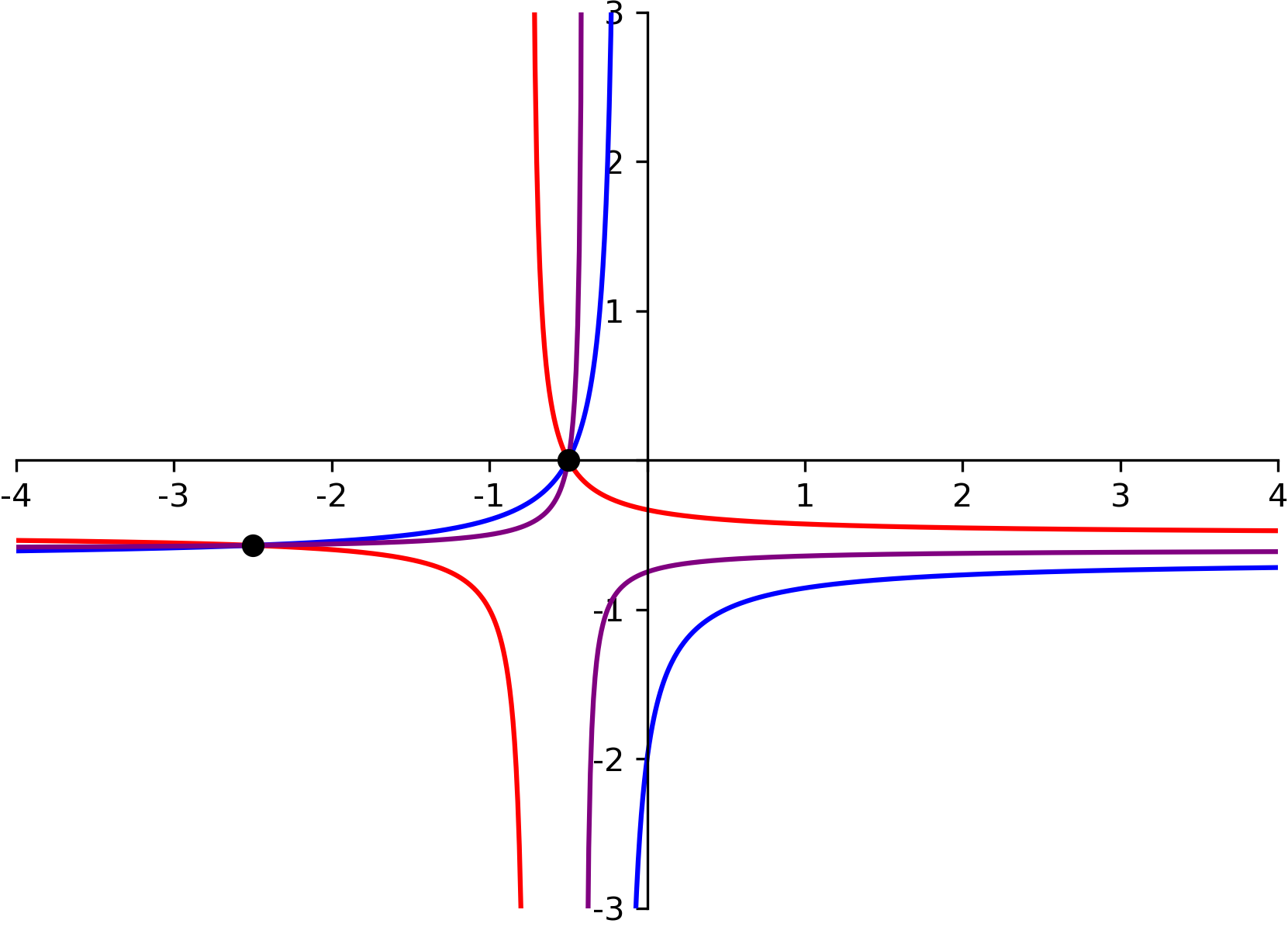}
        \includegraphics[width=0.45\linewidth]{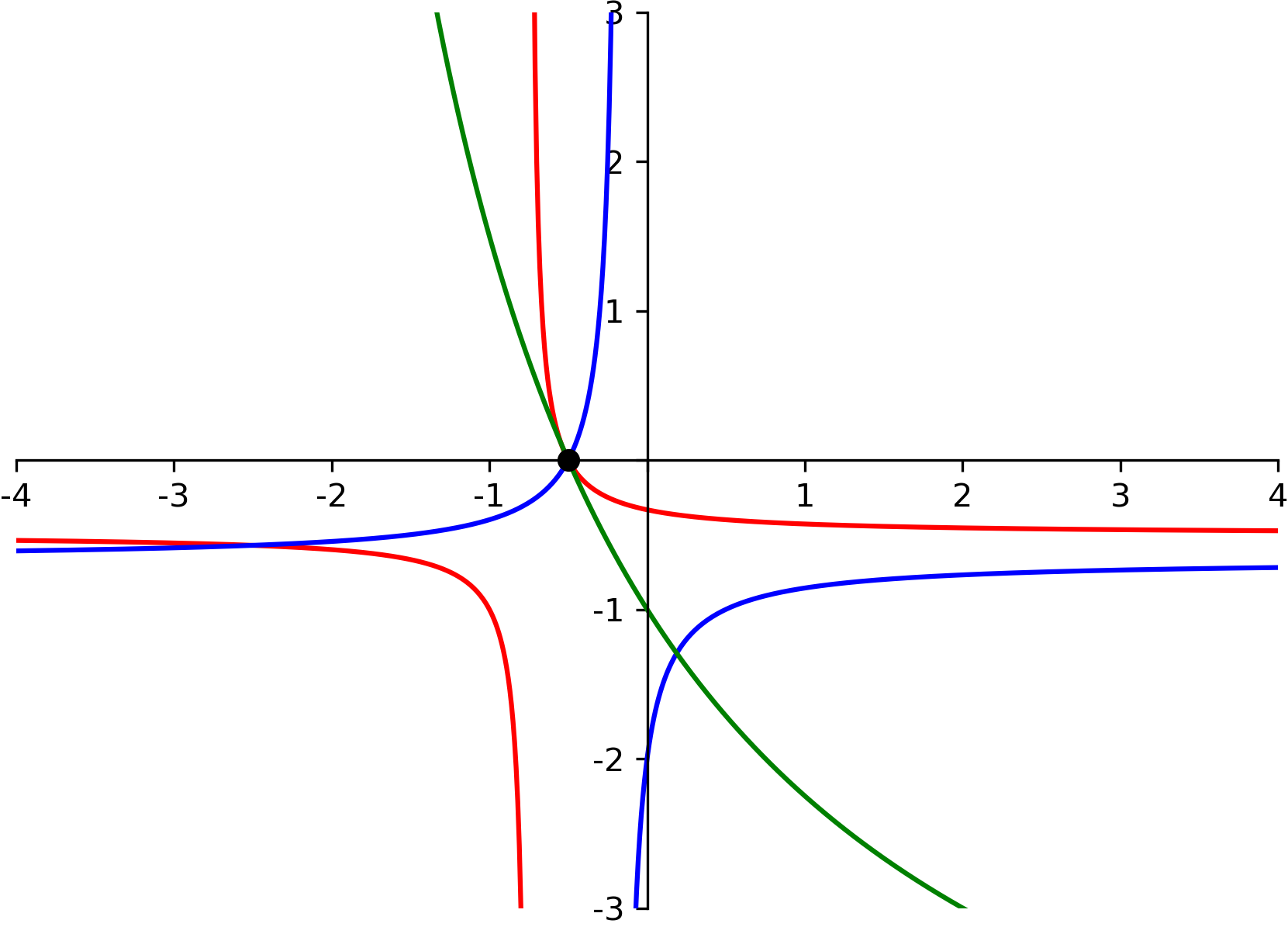}
        \caption{The red, blue, purple, and green quadrics are the vanishing loci of $f_0$, $f_1$, $f_2$, and $f_2'$, respectively. }
        \label{fig:counterex}
\end{figure}
\end{example}

% \textcolor{blue}{Write down \\
% - quadrics (projective and affine),  maybe with pictures? \\
% - explanation of how the quadrics intersect \\
% - $det (T_y)_{ij}$? \\
% - the Euler characteristic of $V(q_0,q_1,q_2)$ \\
% - the Euler charasteristic overall (i.e. the sum) \\
% - other observations that Serkan made}

\section{Euler stratification for $\PP^1 \times \PP^1 \times \PP^1$} \label{sec:P1 x P1 x P1}
In this section we present the complete Euler stratification for the hypersurface family  $Y_{W,1}$. %Our starting point is the following result.
%\begin{theorem} \label{thm:Fevola+Matsubara-Heo}
%\cite[Theorem 2.2]{FM24}
%Let $f(x,z) = f_0(x) + \sum_{k=1}^n z_k f_k(x)$ be a Laurent polynomial where  $f_0,\ldots, f_n \in \CC[x_1^{\pm}, \ldots, x_m^{\pm}]$. Then
%$$ \chi\left(V(f) \cap (\CC^*)^{m+n}\right) = (-1)^n \chi\left(V(f_0f_1\cdots f_n) \cap (\CC^*)^m\right).$$ 
%\end{theorem} 
The first part of Theorem \ref{thm:main} states that this stratification is determined only by the subsets of the factors of the principal $A$-determinant $E_A$ that vanish simultaneously on $W$. Proposition \ref{prop:factors 2x2x2} gives us the precise list of these subsets of the factors. Hence, we know the resulting $41$ strata exactly. On the other hand, we wish to present a geometric interpretation of these strata. As we will show below, the geometry is determined by the various types of ways the two quadrics $f_0$ and $f_1$ from \eqref{eq:f} intersect in $(\CC^*)^2$. The complete list of the possibilities which we determine below reproduce the $41$ strata. One advantage of this approach is the determination of all distinct realizations of the vanishing patterns we have listed in Figure \ref{fig:222_factor_relations} as A through H. For instance, scenario A, where exactly one $2$-minor vanishes, is realized up to symmetry by three different choices of pairs of quadrics with distinct intersection patterns. The figure below shows these three realizations. In all cases, the ML degree is equal to $5$.

\begin{figure}[ht] 
\centering
\captionsetup[subfigure]{labelformat=empty}

\begin{subfigure}{0.32\linewidth}
    \centering
    \includegraphics[width=\linewidth]{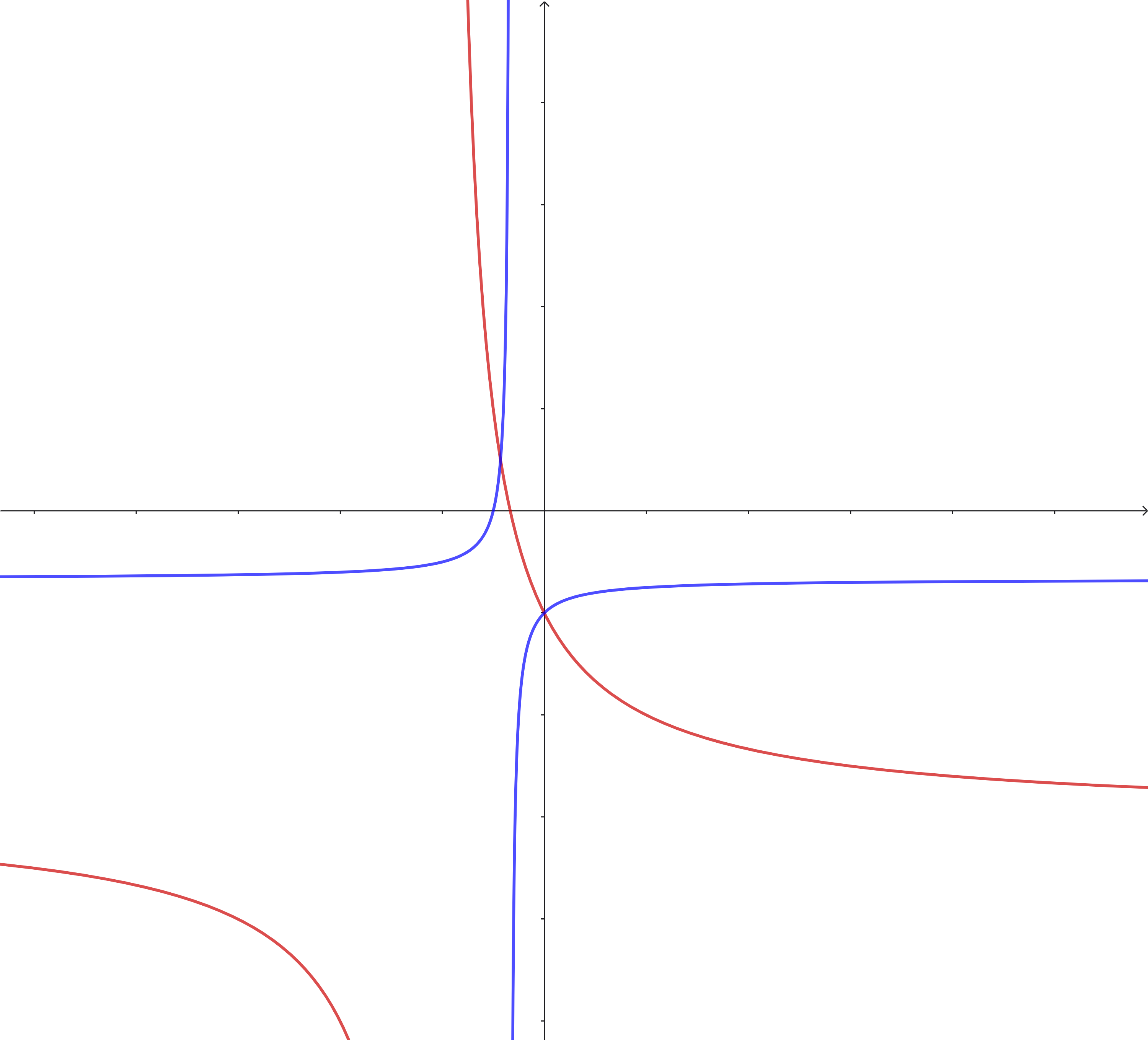}
    \caption{}
\end{subfigure}
\hfill
\begin{subfigure}{0.32\linewidth}
    \centering
    \includegraphics[width=\linewidth]{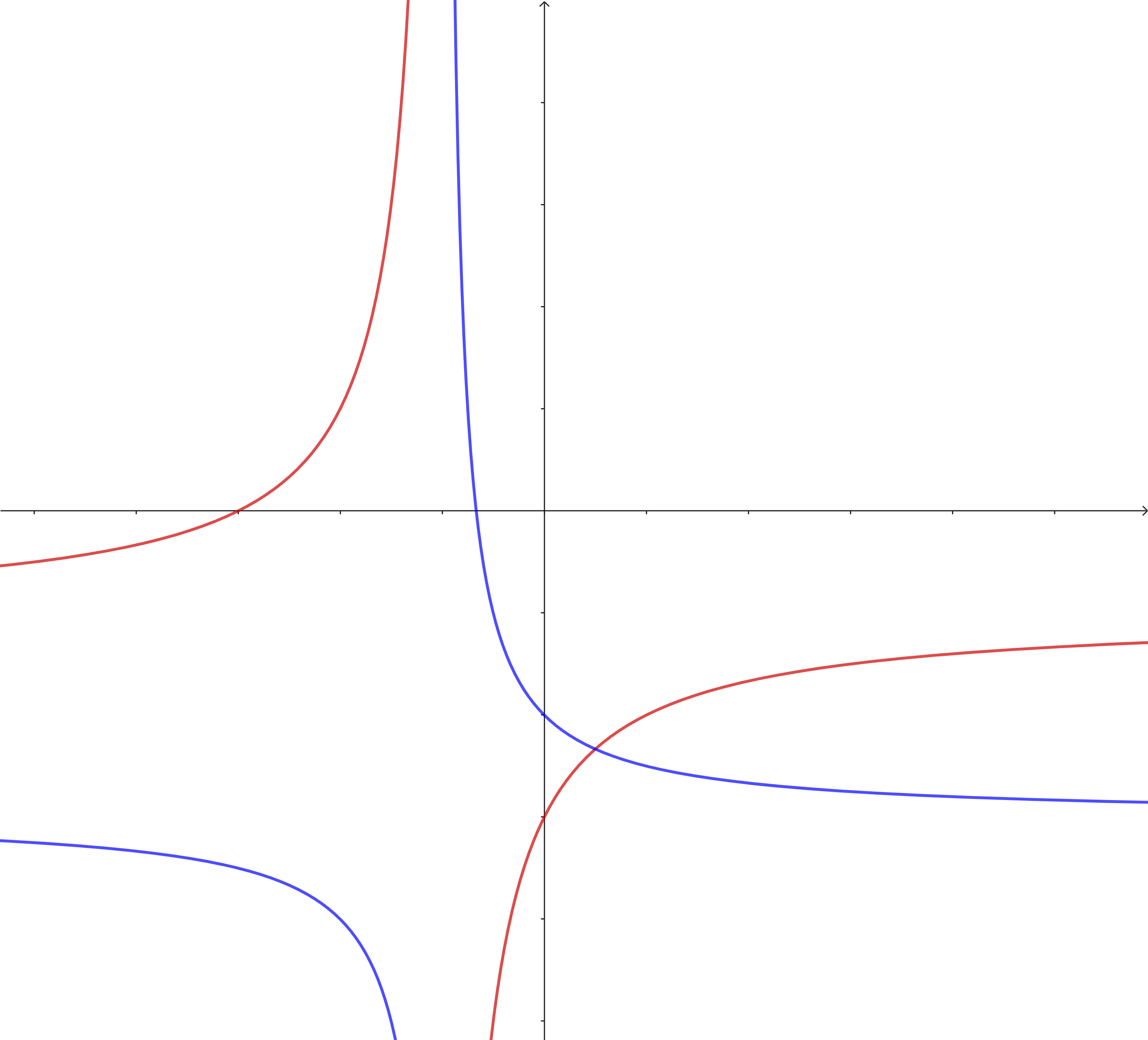}
    \caption{}
\end{subfigure}
\hfill
\begin{subfigure}{0.32\linewidth}
    \centering
    \includegraphics[width=\linewidth]{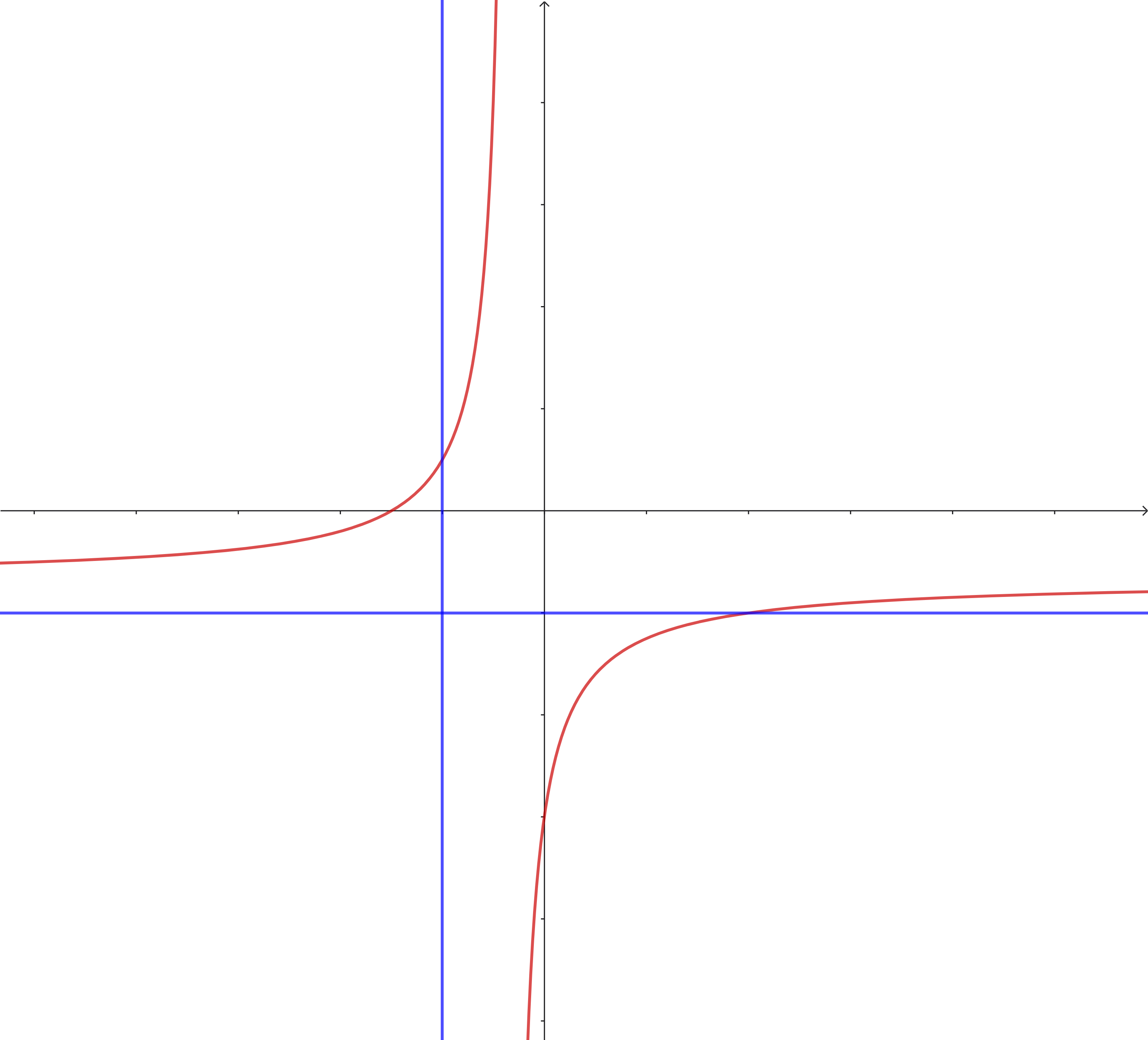}
    \caption{}
\end{subfigure}
\caption{The three possible intersection patterns described by case A.}
\end{figure}

We start with a corollary to Theorem \ref{thm:Fevola+Matsubara-Heo} relevant here.
\begin{corollary} \label{cor: quadrics}
 Let $f_W(x_1,y_1,z_1,\ldots z_n)$ be the defining polynomial of $Y_{W,n} \subset (\CC^*)^{2+n}$ and let $f_0, f_1, \ldots, f_n \in \CC[x_1^{\pm}, y_1^{\pm}]$ be as in \eqref{eq:f}. Then
 $$ \chi(Y_{W,n}) = (-1)^n \chi \left(V(f_0 f_1 \cdots f_n)\right),$$
where $V(f_0f_1\cdots f_n) = \bigcup_{i=0}^n V(f_i)$ is the union of $n+1$ quadric curves in $(\CC^*)^2$.
\end{corollary}
\noindent For $\PP^1 \times \PP^1 \times \PP^1$
we have two quadrics $Q_0$ and $Q_1$ in $(\CC^*)^2$ defined by 
$$ f_0 = w_{000} + w_{100}x + w_{010}y + w_{110} xy \,\, \mbox{   and    }  \,\,
f_1 = w_{001} + w_{101}x + w_{011}y + w_{111} xy.$$ 
\noindent Therefore, we need to characterize $\chi(Q_0 \cup Q_1)$ as 
$W \in (\CC^*)^8$ varies. 
\begin{lemma}
 Let $g(x,y) = a + bx + cy + dxy$ where $a,b,c,d \in \CC^*$. Then the quadric $Q = V(g) \subset \CC^2$ 
 is nonsingular if and only if $ad -bc \neq 0$. In this situation, $Q$ is a hyperbola and intersects the coordinate axes in two points, namely, in $(0,-a/c)$ and $(-a/b, 0)$. The Euler characteristic of $Q$ in $(\CC^*)^2$ is $\chi(Q) = -2$. If $Q$ is singular, it is the union of the lines $x=-\frac{c}{d}$ and $y=-\frac{b}{d}$ that are parallel to the axes. In this case the Euler characteristic of $Q$ in $(\CC^*)^2$ 
 is $\chi(Q) = -1$.
\end{lemma}
\begin{proof}
  Observe that 
  $$ g(x,y) = \begin{bmatrix} x &  y &  1\end{bmatrix} \begin{bmatrix} 0 & \frac{1}{2}d & \frac{1}{2}b \\ \frac{1}{2}d & 0 & \frac{1}{2} c \\ \frac{1}{2} b & \frac{1}{2} c & a \end{bmatrix} \begin{bmatrix} x \\ y \\ 1\end{bmatrix}.$$
  Then, $Q$ is nonsingular if and only if the middle matrix is nonsingular. The determinant of that matrix is $\frac{1}{4} d (bc - ad)$. If this determinant is zero, $g$ factors as $g = d(x+\frac{c}{d})(y+\frac{b}{d})$. The computation of Euler characteristics follows easily.
\end{proof}
\begin{corollary}
 $Q_0$ is nonsingular if and only if $F_{\bullet \bullet 0} = \det(W_{\bullet \bullet 0}) \neq 0$. Similarly, $Q_1$ is nonsingular if and only if $F_{\bullet \bullet 1} = \det(W_{\bullet \bullet 1}) \neq 0$.   
\end{corollary}
By Theorem \ref{thm:main}, we will compute 
$$\chi(Y_{W,1}) = - \chi(Q_0 \cup Q_1) = - \left(\chi(Q_0) + \chi(Q_1) - \chi(Q_0 \cap Q_1)\right)$$ based only on the possible vanishing patterns of factors of $E_A$. We will determine these possible configurations throughout this section and we will study the effect of these vanishing patterns on $\chi(Q_0)$, $\chi(Q_1)$, and $\chi(Q_0\cap Q_1)$. We use the same notation as in the subsection for $\PP^1 \times \PP^1 \times \PP^1$ in Section \ref{sec:factors}. 

\subsection{Both $Q_0$ and $Q_1$  nonsingular}
Here we are assuming that $F_{\bullet \bullet 0}$ and $F_{\bullet \bullet 1}$, the determinants of the two slices $W_{\bullet \bullet 0}$ and $W_{\bullet \bullet 1}$, are not zero. This means that $\chi(Q_0) = \chi(Q_1) = -2$, and we need to determine $\chi(Q_0 \cap Q_1)$. Hence we need to determine all the possible ways $Q_0$ and $Q_1$ intersect in $(\CC^*)^2$.

Consider 
\begin{equation} \label{eq:combine f_0 and f_1:}
w_{111} f_0 - w_{110}f_1 \, = \, \det(W_{1 \bullet \bullet})x + \det(W_{\bullet 1 \bullet})y + (w_{000}w_{111} - w_{110}w_{001}).
\end{equation}
If $F_{\bullet 1 \bullet} \neq 0$, we can solve for $y$ in \eqref{eq:combine f_0 and f_1:}. Plugging this into $f_0$ gives
$$ G_1(x) = \det(W_{1 \bullet \bullet}) x^2 + [(w_{000}w_{111} - w_{110}w_{001}) - (w_{010}w_{101} -w_{100}w_{011})]x + \det(W_{0 \bullet \bullet}) = 0.$$
Similarly, if $F_{1 \bullet \bullet} \neq 0$, we can solve for $x$ in \eqref{eq:combine f_0 and f_1:}. Plugging this into $f_0$ gives
$$ G_2(y) = \det(W_{\bullet  1 \bullet}) y^2 + [(w_{000}w_{111} - w_{110}w_{001}) + (w_{010}w_{101} -w_{100}w_{011})]y + \det(W_{\bullet  0 \bullet}) = 0.$$
The discriminants of both $G_1(x)$ and $G_2(y)$ are equal to the hyperdeterminant $H$. 
\begin{lemma} \label{lem:two intersection points}
The two nonsingular quadrics $Q_0$ and $Q_1$ intersect in two distinct points in $\CC^2$ if $F_{\bullet 1 \bullet} \neq 0$,  
$F_{1 \bullet \bullet} \neq 0$, and $H \neq 0$. In this case, they intersect in a point where $x=0$ if and only if $F_{0 \bullet \bullet} = 0$. Similarly, they intersect in a point where $y=0$ if and only if $F_{\bullet 0  \bullet} = 0$. 
\end{lemma}
\begin{proof}
If the three polynomials do not vanish, then $G_1(x)$ (and $G_2(y)$) has two distinct roots $x_1$ and $x_2$, and one can solve for the $y$ coordinate of the two intersection points using \eqref{eq:combine f_0 and f_1:}. Under these conditions, $G_1(x) =0$ implies that one its roots will be $0$ if and only if $\det(W_{0 \bullet \bullet}) = F_{0 \bullet \bullet} =0$. 
Similarly, $G_2(y) =0$ implies that one its roots will be $0$ if and only if $\det(W_{\bullet 0 \bullet}) = F_{ \bullet 0 \bullet} =0$.
\end{proof}
\subsubsection*{\underline{$F_{\bullet 1 \bullet} \neq 0$, $F_{1 \bullet \bullet} \neq 0$, and $H \neq 0$}}
There are only two factors left to consider and the following captures the implications of their vanishings.
\begin{proposition} \label{prop: two intersection points}
Suppose $Q_0$ and $Q_1$ are nonsingular quadrics and $F_{\bullet 1 \bullet} \neq 0$, $F_{1 \bullet \bullet} \neq 0$, and $H \neq 0$. 
\begin{itemize}
    \item If neither $F_{0 \bullet \bullet}$ nor $F_{\bullet 0 \bullet}$ vanishes, then $\chi(Y_{W,1}) = 6$.
    \item If exactly one of $F_{0 \bullet \bullet}$ and $F_{\bullet 0 \bullet}$ vanishes, then $\chi(Y_{W,1}) = 5$.
    \item If both $F_{0 \bullet \bullet}$ and $F_{\bullet 0 \bullet}$ vanish, then $\chi(Y_{W,1}) = 4$.
\end{itemize}
\end{proposition}
\begin{proof}
By Lemma \ref{lem:two intersection points}, the quadrics intersect in exactly two points in $(\CC^*)^2$ if the constant terms of $G_1(x)$ and $G_2(y)$ do not vanish. In this case, $\chi(Q_0 \cap Q_1)= 2$ and hence
$\chi(Q_0 \cup Q_1) = -6$. If exactly one of these constant terms vanish, exactly one of these points is on a coordinate axis. This means $\chi(Q_0 \cap Q_1)= 1$ and hence $\chi(Q_0 \cup Q_1) = -5$. If both constant terms are zero, both intersection points are on the coordinate axes. Therefore, $\chi(Q_0 \cap Q_1)= 0$ and  $\chi(Q_0 \cup Q_1) = -4$.
\end{proof}
We point out that the second case above, where $\chi(Y_{W,1})=5$, corresponds to exactly one of the minors of $W$ vanishing. This is the case A in Figure \ref{fig:222_factor_relations}. The third case, where $\chi(Y_{W,1}) = 4$, corresponds to exactly two minors which form a hook vanishing. This is the case C in Figure \ref{fig:222_factor_relations}.

\subsubsection*{\underline{$F_{\bullet 1 \bullet} \neq 0$, $F_{1 \bullet \bullet} \neq 0$, and $H = 0$}}
The vanishing of the hyperdeterminant implies that the quadrics $Q_0$ and $Q_1$ are tangent to each other.
\begin{lemma} \label{lem:tangency}
Suppose the quadrics $Q_0$ and $Q_1$ are nonsingular and $F_{\bullet 1 \bullet} \neq 0$ and $F_{1 \bullet \bullet} \neq 0$. If the hyperdeterminant vanishes, $Q_0$ and $Q_1$ intersect in $\CC^2$ in a single point of multiplicity two. Furthermore,
\begin{enumerate}
\item If neither $F_{0 \bullet \bullet}$ nor $F_{\bullet 0 \bullet}$ vanishes, then the unique intersection point is in $(\CC^*)^2$.
    \item If exactly one of $F_{0 \bullet \bullet}$ and $F_{\bullet 0 \bullet}$ vanishes, then the unique intersection point is on one of the coordinate axes. 
%    \item If both $F_{0 \bullet \bullet}$ and $F_{\bullet 0 \bullet}$ vanish, then the quadrics $Q_0$ and $Q_1$ coincide. 
\end{enumerate}
\end{lemma}
\begin{proof} Under the hypotheses, $G_1(x)$ and $G_2(y)$ have a double root. This means that $Q_0$ and $Q_1$ intersect in a single point of tangency. If the constant terms of $G_1(x)$ and $G_2(y)$, namely, $F_{0\bullet \bullet}$ and $F_{\bullet 0 \bullet}$ do not vanish, this intersection point is in $(\CC^*)^2$.
If exactly one of them vanishes, then $x=0$ is a double root of $G_1(x)$ or 
$y=0$ is a double root of $G_2(y)$.
\end{proof}
The alert reader would realize that we have omitted the case when  
$F_{0 \bullet \bullet} = F_{\bullet 0 \bullet} =0$. This corresponds to the case when the minors in a hook and the hyperdeterminant vanish. However, this forces the minors in the cubic frame containing the hook all vanish. The extra minors are $F_{\bullet 1 \bullet}$ and $F_{1 \bullet \bullet}$ which we assumed to be nonzero. 
\begin{proposition} \label{prop:tangency}
Suppose $Q_0$ and $Q_1$ are nonsingular quadrics and $F_{\bullet 1 \bullet} \neq 0$, $F_{1 \bullet \bullet} \neq 0$, and $H = 0$. 
\begin{itemize}
    \item If neither $F_{0 \bullet \bullet}$ nor $F_{\bullet 0 \bullet}$ vanishes, then $\chi(Y_{W,1}) = 5$.
    \item If exactly one of $F_{0 \bullet \bullet}$ and $F_{\bullet 0 \bullet}$ vanishes, then $\chi(Y_{W,1}) = 4$.
\end{itemize}    
\end{proposition}
\begin{proof}
By Lemma \ref{lem:tangency}, the intersection $Q_0 \cap Q_1$ consists of a single point. In the first case, this point is in $(\CC^*)^2$. Therefore  $\chi(Q_0 \cap Q_1) = 1$ and $\chi(Y_{W,1}) = 5$. In the second case $Q_0 \cap Q_1 = \emptyset$ in $(\CC^*)^2$. Therefore  $\chi(Q_0 \cap Q_1) = 0$ and $\chi(Y_{W,1}) = 4$.
\end{proof}
The first case above, where $\chi(Y_{W,1}) = 5$, corresponds to exactly the hyperdeterminant vanishing. This is the case B in Figure \ref{fig:222_factor_relations}. The second case where $\chi(Y_{W,1}) = 4$ corresponds to one minor plus the hyperdeterminant vanishing. This is the case E in Figure \ref{fig:222_factor_relations}.

\subsubsection*{\underline{Exactly one of $F_{\bullet 1 \bullet}$ and $F_{1 \bullet \bullet}$ vanishes, and $H \neq 0$}}
Without loss of generality, we will assume $F_{\bullet 1 \bullet} \neq 0$ but $F_{1 \bullet \bullet}=0$ since the other case leads to similar conclusions by symmetry (via exchanging the roles of the variables $x$ and $y$). 

\begin{lemma} \label{lem:unique-point}
Suppose the quadrics $Q_0$ and $Q_1$ are nonsingular, and $F_{\bullet 1 \bullet} \neq 0$, $F_{1 \bullet \bullet} =0$, and $H \neq 0$. Then $Q_0$ and $Q_1$ intersect in a unique point in $\CC^2$. Furthermore, 
\begin{enumerate}
\item If neither $F_{0 \bullet \bullet}$ nor $F_{\bullet 0 \bullet}$ vanishes, then this unique intersection point is in $(\CC^*)^2$. 
    \item If exactly one of $F_{0 \bullet \bullet}$ and $F_{\bullet 0 \bullet}$ vanishes, then the unique intersection point is on one of the coordinate axes. 
\end{enumerate}
\end{lemma}
\begin{proof}
When only $F_{1 \bullet \bullet}$ vanishes, then in \eqref{eq:combine f_0 and f_1:} we can uniquely solve for $y$. In addition, $H \neq 0$ implies that $G_1(x)$ has a unique root. Therefore $Q_0$ and $Q_1$ intersect in a unique point in $\CC^2$. If $F_{0 \bullet \bullet}$ vanishes, then $x=0$ is the unique root of $G_1(x)$. If $F_{\bullet 0 \bullet}$ vanishes, we see that 
$w_{000}w_{111} - w_{110}w_{001}$ is in
$\langle F_{1 \bullet \bullet}, F_{\bullet 0 \bullet} \rangle \, : \ \left( \prod w_{ijk}\right)^\infty$. This means that the constant term in $\eqref{eq:combine f_0 and f_1:}$ is zero. Therefore $y=0$. 
\end{proof}
Again, we have not considered the case $F_{0 \bullet \bullet} = F_{\bullet 0 \bullet} =0$, since together with $F_{1 \bullet \bullet} = 0$ we have a square cup with all minors vanishing. However, this implies that the missing minor in the cubic frame containing this square cup, namely, $F_{\bullet 1 \bullet}$, as well as $H$ vanish. However, we assumed that these polynomials are not zero.

\begin{proposition} \label{prop: f1dotdot=0}
Suppose $Q_0$ and $Q_1$ are nonsingular quadrics, and $F_{\bullet 1 \bullet} \neq 0$, $F_{1 \bullet \bullet} =0$, and $H \neq 0$. 
\begin{itemize}
    \item If neither $F_{0 \bullet \bullet}$ nor $F_{\bullet 0 \bullet}$ vanishes, then $\chi(Y_{W,1}) = 5$.
    \item If exactly one of $F_{0 \bullet \bullet}$ and $F_{\bullet 0 \bullet}$ vanishes, then $\chi(Y_{W,1}) = 4$.
\end{itemize}    
\end{proposition}
\begin{proof}
 The proof is identical to the proof of Proposition \ref{prop:tangency}. 
\end{proof}
The first case above, where $\chi(Y_{W,1}) = 5$, corresponds to exactly one minor of $W$ vanishing. 
This is the case A in Figure \ref{fig:222_factor_relations}. The second case, where $\chi(Y_{W,1})=4$, corresponds to either the minors in a hook or in a mirror vanishing. These are the cases C and D in Figure \ref{fig:222_factor_relations}, respectively.

\subsubsection*{\underline{Exactly one of $F_{\bullet 1 \bullet}$ and $F_{1 \bullet \bullet}$ vanishes, and $H = 0$}}
Again by symmetry we assume that $
F_{\bullet 1 \bullet} \neq 0$ and $F_{1 \bullet \bullet} = H = 0$. We immediately see that $G_1(x) = F_{0 \bullet \bullet}$. Therefore, if $F_{0 \bullet \bullet} \neq 0$, the quadrics $Q_0$ and $Q_1$ do not intersect. Otherwise, the minors in a mirror and $H$ vanish. This implies that the minors in a cubic frame containing the mirror also vanish. But this contradicts our assumption that $Q_0$ and $Q_1$ are nonsingular, as well as 
$F_{\bullet 1 \bullet} \neq 0$. As a result, we conclude the following. 
\begin{proposition} \label{prop: minor+hyper} Suppose the quadrics $Q_0$ and $Q_1$ are nonsingular, exactly one of $F_{\bullet 1 \bullet}$ and $F_{1 \bullet \bullet}$ vanishes, and $H=0$ and the remaining minors of $W$  do not vanish. Then $Q_0$ and $Q_1$ do not intersect in $\CC^2$ and $\chi(Y_{W,1}) = 4$. 
\end{proposition}
The only case above, where $\chi(Y_{W,1}) = 4$, corresponds to exactly one minor of $W$ and the hyperdeterminant vanishing. This is the case $E$ in Figure \ref{fig:222_factor_relations}. 

\subsubsection*{\underline{$F_{\bullet 1 \bullet} = F_{1 \bullet \bullet} =0$ and $H \neq 0$}}
\begin{lemma} \label{lem:hook}
Suppose the quadrics $Q_0$ and $Q_1$ are nonsingular. If $F_{\bullet 1 \bullet} = F_{1 \bullet \bullet} = 0$
and $H \neq 0$, then $Q_0$ and $Q_1$ do not intersect in $\CC^2$ if and only if $F_{0 \bullet \bullet} \neq 0$.
\end{lemma}
\begin{proof} Under the above hypotheses, the quadrics do not intersect if and only if the constant term $w_{000}w_{111} - w_{110}w_{001}$ in \eqref{eq:combine f_0 and f_1:} is not equal to zero. This constant term is 
the determinant of the first and third columns of the matrix
$\begin{bmatrix} 
w_{000} & w_{010} & w_{110} \\
w_{001} & w_{001} & w_{111}
\end{bmatrix}$. The determinant of the last two columns is $F_{\bullet 1 \bullet} = 0$. Since the determinant of the first two columns is $F_{0 \bullet \bullet}$, we conclude that $w_{000}w_{111} - w_{110}w_{001} \neq 0$ if and only if $F_{0 \bullet \bullet} \neq 0$.
\end{proof}
\begin{proposition} \label{prop: hook}
Suppose the quadrics $Q_0$ and $Q_1$ are nonsingular, both $F_{\bullet 1 \bullet}$ and $F_{1 \bullet \bullet}$ vanish, but $H$, $F_{0 \bullet \bullet}$ and $F_{\bullet 0 \bullet}$ do not vanish. Then $Q_0$ and $Q_1$ do not intersect in $\CC^2$ and $\chi(Y_{W,1}) = 4$.
\end{proposition}
The case in the above proposition, where $\chi(Y_{W,1})=4$, has the minors in a hook vanishing. This is the case 
C in Figure \ref{fig:222_factor_relations}. 
Note that by Lemma \ref{lem:hook}, if $F_{0 \bullet \bullet} =0$ the the quadrics intersect. However, the three vanishing minors form a square cup, and  this implies the vanishing of all the minors in the cubic frame containing the square cup and the hyperdeterminant. Since we assumed $H \neq 0$, we do not consider this case here.  

\subsubsection*{\underline{$F_{\bullet 1 \bullet} = F_{1 \bullet \bullet} = H =0$}}
\begin{proposition} \label{prop: frame+hyper}
Suppose the quadrics $Q_0$ and $Q_1$ are nonsingular and $F_{\bullet 1 \bullet} = F_{1 \bullet \bullet} = H =0$. This implies that $F_{0 \bullet \bullet} = F_{\bullet 0 \bullet} =0$. In this case, $Q_0 = Q_1$ and $\chi(Y_{W,1}) = 2$. 
\end{proposition}
\begin{proof}
It is easy to compute that 
the radical of the ideal $\langle H, F_{\bullet 1 \bullet}, F_{1 \bullet \bullet} \rangle \, : \, \left( \prod w_{ijk}\right)^\infty$ contains $F_{0 \bullet \bullet}$ and $F_{\bullet 0 \bullet}$.  This implies that the coefficients of $f_1$ is a constant multiple of the coefficients of $f_0$. Hence, $Q_0 =Q_1$ and the result follows.
\end{proof}
This last case corresponds to the minors in a cubic frame and the hyperdeterminant vanishing. This is the case G in Figure \ref{fig:222_factor_relations}.

\subsection{$Q_0$ nonsingular and $Q_1$ singular}
Now we assume that $F_{\bullet \bullet 0} \neq 0$ and $F_{\bullet \bullet 1} =0$. In other words, $Q_0$ is nonsingular but $Q_1$ is the union of the lines $L_1$ and $L_2$ defined by $x = -\frac{w_{011}}{w_{111}}$ and $y=-\frac{w_{101}}{w_{111}}$, respectively. This means that $\chi(Q_0) = -2$ and $\chi(Q_1) = -1$. Again, we need to determine all the possible ways $Q_0$ and $Q_1$ intersect in $(\CC^*)^2$.

\subsubsection*{\underline{$F_{\bullet 1 \bullet} \neq 0$, $F_{1 \bullet \bullet} \neq 0$, and $H \neq 0$}}
We start out with pointing out that Lemma \ref{lem:two intersection points} remains valid even if we relax the condition on the nonsingularity of $Q_1$. Therefore we get the analog of Proposition \ref{prop: two intersection points}.
\begin{proposition} \label{prop: nonsingular+singular}
Suppose $Q_0$ is a nonsingular quadric and $Q_1$ is a singular quadric. Moreover, suppose that $F_{\bullet 1 \bullet} \neq 0$, $F_{1 \bullet \bullet} \neq 0$, and $H \neq 0$. 
\begin{itemize}
    \item If neither $F_{0 \bullet \bullet}$ nor $F_{\bullet 0 \bullet}$ vanishes, then $\chi(Y_{W,1}) = 5$.
    \item If exactly one of $F_{0 \bullet \bullet}$ and $F_{\bullet 0 \bullet}$ vanishes, then $\chi(Y_{W,1}) = 4$.
    \item If both $F_{0 \bullet \bullet}$ and $F_{\bullet 0 \bullet}$ vanish, then $\chi(Y_{W,1}) = 3$.
\end{itemize}    
\end{proposition}
\begin{proof}
The proof is identical to the proof of 
Proposition \ref{prop: two intersection points}, except  
now $\chi(Q_0) + \chi(Q_1) =~-3$.
\end{proof}
We again point out that the first case above, where $\chi(Y_{W,1})=5$, corresponds to exactly one of the minors of $W$ vanishing. This is the case A in Figure \ref{fig:222_factor_relations}. The second case where $\chi(Y_{W,1}) = 4$ corresponds to exactly two minors which form a hook vanishing. This is the case C in Figure \ref{fig:222_factor_relations}.
Finally, the third case where $\chi(Y_{W,1})=3$ corresponds to three minors meeting at a vertex of $W$ vanishing. This is the case F.

\subsubsection*{\underline{$F_{\bullet 1 \bullet} \neq 0$, $F_{1 \bullet \bullet} \neq 0$, and $H = 0$}}
Now Lemma \ref{lem:tangency} stays valid except its last statement. The two quadrics intersect in a single point of multiplicity two: the intersection point of $L_1$ and $L_2$ lies on $Q_0$. Moreover, if neither $F_{0\bullet \bullet}$ nor $F_{\bullet 0 \bullet}$ vanishes, this intersection point is in $(\CC^*)^2$. If at least one of these vanishes, we will have the minors of a hook and hyperdeterminant vanish. But then all minors in the cubic frame containing this hook, including $F_{\bullet \bullet 0}$, are equal to zero. But we assumed that $Q_0$ is nonsingular. 
\begin{proposition} \label{prop: one-singular + hyper}
Suppose $Q_0$ is a nonsingular quadric and $Q_1$ is a singular quadric. If only the hyperdeterminant $H$ vanishes, then $\chi(Y_{W,1}) = 4$.    
\end{proposition}
\begin{proof}
Based on the observations above we conclude that $\chi(Q_0 \cap Q_1) = -1$.
\end{proof}
This case, where $\chi(Y_{W,1}) =4$, corresponds to exactly one minor of $W$
and the hyperdeterminant $H$ vanishing. 
This is the case E in Figure \ref{fig:222_factor_relations}.

\subsubsection*{\underline{Exactly one of $F_{\bullet 1 \bullet}$ and $F_{1 \bullet \bullet}$ vanishes, and $H \neq 0$}}
Without loss of generality, we will assume $F_{\bullet 1 \bullet} \neq 0$ but $F_{1 \bullet \bullet}=0$. This time, Lemma \ref{lem:unique-point} stays valid with a small modification to its last statement. As before, $Q_0$  and $Q_1$ intersect in a unique point. This point is the intersection of $Q_0$ and $L_1$, and $L_2$ does not intersect $Q_0$ at all. The intersection point is in $(\CC^*)^2$ if
neither $F_{0\bullet \bullet}$ nor $F_{\bullet 0 \bullet}$ vanishes. Note that if $F_{0 \bullet \bullet}$ vanishes we have the three minors of a square cup vanishing. And this forces $F_{\bullet \bullet 0}=0$, implying that $Q_0$ is singular. But we are assuming $Q_0$ is nonsingular. 
On the other hand, if $F_{\bullet 0 \bullet}$ vanishes, $Q_0$ and $Q_1$ intersect in a unique point on $y=0$. Hence, we get the following result. 

\begin{proposition} \label{prop: one-singular hook}
Suppose $Q_0$ is a nonsingular quadric and $Q_1$ is a singular quadric. Moreover, suppose that $F_{\bullet 1 \bullet} \neq 0$, $F_{1 \bullet \bullet} =0$, and $H \neq 0$. 
\begin{itemize}
    \item If neither $F_{0 \bullet \bullet}$ nor $F_{\bullet 0 \bullet}$ vanishes, then $\chi(Y_{W,1}) = 4$.
    \item If $F_{\bullet 0 \bullet}$ vanishes, then $\chi(Y_{W,1}) = 3$.
\end{itemize}    
\end{proposition}
\begin{proof}
The first case corresponds to the configuration where $Q_0$ and $Q_1$ intersect in a unique point in $(\CC^*)^2$. Therefore $\chi(Q_0 \cap Q_1) = -1$. In the second case, this unique intersection point is on $y=0$, and hence $\chi(Q_0 \cap \Q_1) = 0$.   
\end{proof}
The first case, where $\chi(Y_{W,1}) = 4$, corresponds to the minors of a hook vanishing. This is the case C in Figure \ref{fig:222_factor_relations}. 
The second case, where $\chi(Y_{W,1}) = 3$, corresponds to three minors touching a vertex of $W$ vanishing. This is the case F.

The next possibility is where exactly one of $F_{ \bullet 1 \bullet}$ and $F_{1 \bullet \bullet}$ as well as $H$ vanish. However, both lead to a configuration where the minors of a hook and the hyperdeterminant vanish. This forces the minors of the cubic frame containing the hook to vanish. However, one of these minors is $F_{\bullet \bullet 0}$, and therefore $Q_0$ would be singular; a contradiction to our assumption. So we skip this case.

\subsubsection*{\underline{$F_{\bullet 1 \bullet} = F_{1 \bullet \bullet} =0$ and $H \neq 0$}}
As in Lemma \ref{lem:hook}, if the remaining two minors do not vanish, then $Q_0$ and $Q_1$ do not intersect 
in $\CC^2$. This is equivalent to $L_1$ and $L_2$ being the asymptotes of $Q_0$. 
\begin{proposition} \label{prop: corner}
Suppose the quadric $Q_0$ is nonsingular and $Q_1$ is singular. Further suppose that both $F_{\bullet 1 \bullet}$ and $F_{1 \bullet \bullet}$ vanish, but $H$, $F_{0 \bullet \bullet}$ and $F_{\bullet 0 \bullet}$ do not vanish. Then $Q_0$ and $Q_1$ do not intersect in $\CC^2$ and $\chi(Y_{W,1}) = 3$.
\end{proposition}
The above case, where $\chi(Y_{W,1}) = 3$, corresponds to three minors touching a common vertex of $W$ vanishing. This is the case F in Figure \ref{fig:222_factor_relations}.

Finally, we see that we need to skip the possibility where $F_{\bullet 1 \bullet} = F_{1 \bullet \bullet} = H = 0$, since as before, this forces $Q_0$ to be singular. 

\subsection{Both $Q_0$ and $Q_1$ singular}
In the last few cases, we need to consider both $Q_0$ and $Q_1$ as unions of lines: $Q_0 = K_1 \cup K_2 = 
\{x = -\frac{w_{010}}{w_{110}} \} \cup
\{y = -\frac{w_{100}}{w_{110}} \}$ and 
$Q_1 = L_1 \cup L_2 = 
\{x = -\frac{w_{011}}{w_{111}} \} \cup
\{y = -\frac{w_{101}}{w_{111}} \}$. 
Based on our work in Section \ref{sec:factors} and our discussion so far in this section,  we note that as soon as a single extra factor of the principal $A$-determinant $E_A$ vanishes, then all minors of a cubic frame containing the minors $F_{\bullet \bullet 0}$ and $F_{\bullet \bullet 1}$, as well as the hyperdeterminant $H$ vanish. 
If yet another one of the remaining two factors vanish, then all factors of $E_A$ vanish. Therefore, we just have to deal with these three possibilities.
We summarize the corresponding result in the following.
\begin{proposition} \label{prop: both-singular}
Suppose both quadrics $Q_0$ and $Q_1$ are singular. 
\begin{itemize}
    \item If no other factor of $E_A$ vanishes, then $\chi(Y_{W,1}) = 4$.
    \item If $F_{0 \bullet \bullet} = F_{1 \bullet \bullet} = H =0$ or $F_{\bullet 0 \bullet} = F_{ \bullet 1 \bullet} = H =0$  and no other factor vanishes, then $\chi(Y_{W,1}) = 2$.
    \item If all factors of $E_A$ vanish, then $\chi(Y_{W,1})=1$.
\end{itemize}
\end{proposition}
\begin{proof}
In the first case, Lemma \ref{lem:two intersection points} implies that $Q_0$ and $Q_1$ intersect in two points in $(\CC^*)^2$. Since $\chi(Q_0) = \chi(Q_1) = -1$, we conclude that $\chi(Y_{W,1}) = 4$. The second case corresponds to the situation where $K_1 = L_1$ and $K_2 \neq L_2$ or $K_1 \neq L_1$ and $K_2 = L_2$. Since the Euler characteristic of a line of the form $x=a \neq 0$ or $y=b \neq 0$ in $(\CC^*)^2$  is $0$, we get $\chi(Y_{W,1}) = 2$. Finally, if all factors of $E_A$ vanish, this forces 
all coefficients of $f_1$ to be a constant multiple of the corresponding coefficients of $f_0$. Therefore, $Q_0=Q_1$, and $\chi(Y_{W,1}) = 1$.
\end{proof}
Now we put together all the results above in the proof of Theorem \ref{thm:P1 x P1 x P1}.
\begin{proof}[Proof of Theorem \ref{thm:P1 x P1 x P1}:] When no factor of $E_A$ vanishes the ML degree is maximum possible and is equal to $\deg(\mathcal{X}_1) = 6$. 
By the second case of Proposition \ref{prop: two intersection points},
the first case of Proposition \ref{prop:tangency}, the first case of Proposition \ref{prop: f1dotdot=0},
and the first case of Proposition \ref{prop: nonsingular+singular} together with the symmetries of our $2 \times 2 \times 2$ tensor $W$, we see that when exactly a single factor of $E_A$ vanishes, then $\chi(Y_{W,1}) = 5$. By the last case of Proposition \ref{prop: two intersection points},  the second case of Proposition  \ref{prop:tangency},  the second case of Proposition \ref{prop: f1dotdot=0}, Proposition \ref{prop: minor+hyper}, Proposition \ref{prop: hook}, the second case of Proposition \ref{prop: nonsingular+singular}, Proposition \ref{prop: one-singular + hyper}, the first case of Proposition \ref{prop: one-singular hook}, and the first case Proposition \ref{prop: both-singular} together with the symmetries of $W$, we conclude that the two minors in every mirror, every hook, and any minor together with the hyperdeterminant can vanish on their own. This accounts for all $\binom{7}{2}$ pairs of factors of $E_A$. In all of these cases $\chi(Y_{W,1})=4$. The only possibilities when exactly three factors vanish appear in the last case of Proposition \ref{prop: nonsingular+singular}, the second case of Proposition \ref{prop: one-singular hook}, and Proposition \ref{prop: corner}. These are cases when three $2$-minors meeting at a corner of the tensor $W$ vanish. By symmetry there are eight of them and the Euler characteristic is $\chi(Y_{W,1})=3$. 
We do not find a possibility where a subset of four factors of $E_A$ vanish. The only time five factors can vanish simultaneously is given by Proposition \ref{prop: frame+hyper} and the second case of
Proposition \ref{prop: both-singular}. These correspond the $2$-minors in a cubic frame and the hyperdeterminant vanishing. By symmetry there are three such cases and then $\chi(Y_{W,1})=2$. We also do not find a possibility where all but one factor of $E_A$ simultaneously vanish. Finally, when all factors vanish we see that $\chi(Y_{W,1})=1$.
\end{proof}

\section{Realizability of ML degrees} \label{sec:realizability}

In this section we give a positive answer to the question of the realizability of $\mldeg(X_{W,n})$ by proving Theorem \ref{thm:realizability}. Throughout this section, $Q_k$ will denote the quadric curve in $(\CC^*)^2$ defined by $f_k = w_{00k} + w_{10k}x + w_{01k}y + w_{11k} xy$ for $k=0,\ldots,n$ as in \eqref{eq:f}.
In order to prove the result, we   construct a set of scalings $W \in (\CC^*)^{4n + 4}$ whose corresponding arrangement of quadric curves achieve every possible Euler characteristic.
We start with the following simple corollary. 
\begin{corollary}\label{cor: point_formula}
 Let $W \in (\CC^*)^{4n + 4}$ be such that $W$ does not cause any %$2 \times 2 \times 2$ or 
 $2 \times 2 \times 3$ hyperdeterminantal factors of $E_A$ to vanish, then
    we have the following equality:

    $$\mldeg(X_{W,n}) \, =  (-1)^{n+1}\chi(Y_{W,n}) \, = -\left(\sum_{k=0}^n \chi(Q_k) - \sum_{\substack{j,k = 0 \\ j \neq k}}^n \abs{Q_j \cap Q_k}\right) $$
\end{corollary}
\begin{proof}
    By Corollary \ref{cor: quadrics}, 
    we need to compute $\chi(Q_0 \cup Q_1 \cup \cdots \cup Q_n)$. Using inclusion/exclusion and the fact that a triple intersection $Q_i \cap Q_j \cap Q_k$ in $\CC^2$ is nonempty if and only if the corresponding
hyperdeterminant $H_{ijk}$ vanishes proves the result.
\end{proof}

We recall that, 
for $0 \leq j\neq k \leq n$, a choice of $W$ such that $F_{\bullet 1 (j,k)} = F_{0 \bullet (j,k)} = 0$ and no other relations among their coefficients are satisfied gives us two nonsingular quadrics with
$Q_j \cap Q_k = \emptyset$ in the torus.
It follows from Lemma \ref{lem:unique-point} that $Q_j$ and $Q_k$ intersect in a unique point in $\CC^2$ and this point is on $y=0$.
Likewise, if we pick $W$ such that $F_{\bullet 0 (j,k)} = F_{1 \bullet (j,k)} = 0$, then the intersection point is on $x=0$. Figure \ref{fig:degenerate_quadrics} illustrates this situation.

 \begin{figure}[H]
        \centering
        \includegraphics[width=0.4\linewidth]{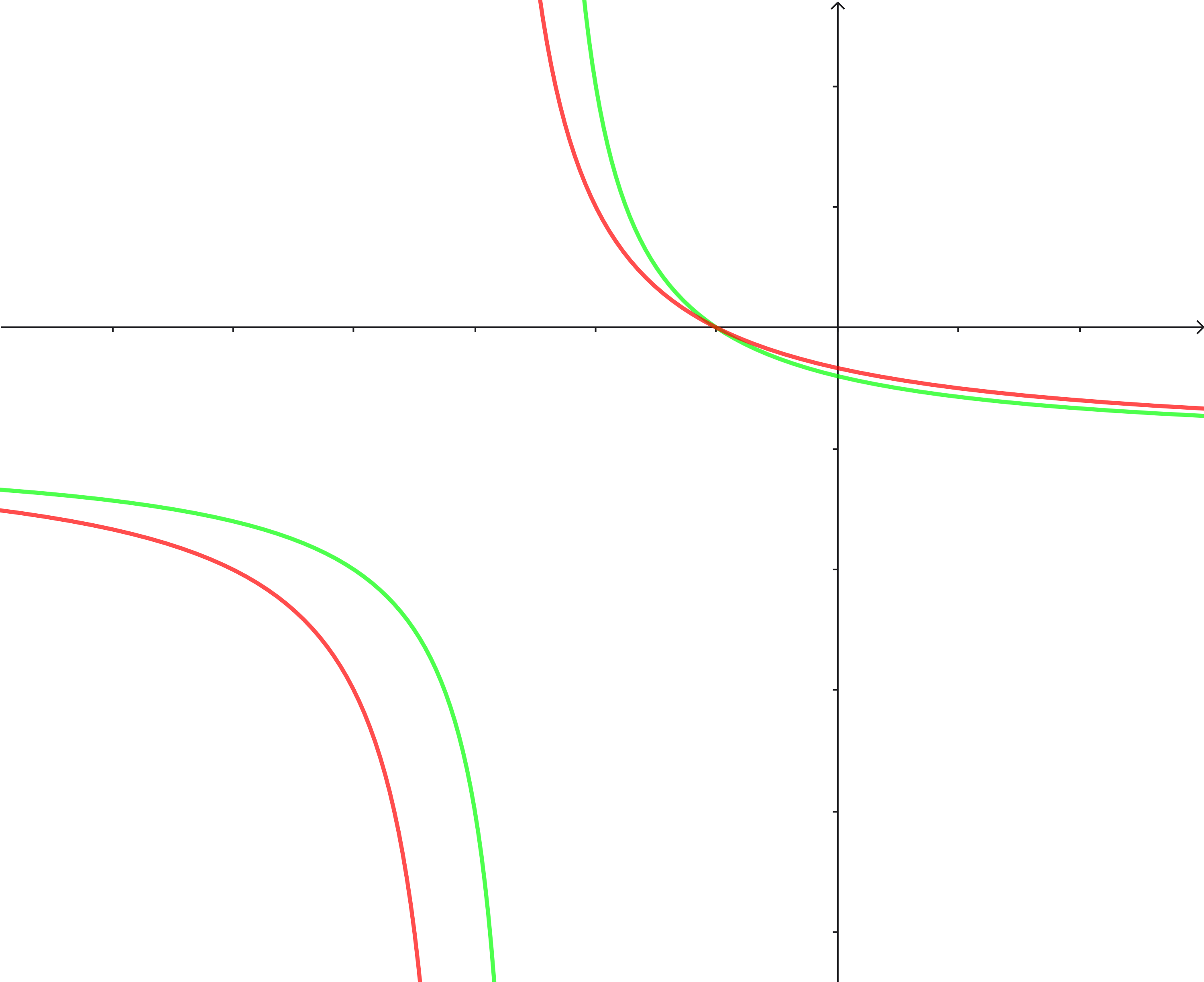}
        \caption{Two nonsingular quadrics $Q_j, Q_k$ where $F_{\bullet 1 (j,k)} = F_{0 \bullet (j,k)} = 0$}
        \label{fig:degenerate_quadrics}
    \end{figure}

\begin{definition}
    The set $\mathrm{altH}(n)$ of alternating hook minors consists of the factors of $E_A$ given by $F_{\bullet 1 (a,a+1)},F_{0 \bullet (a,a+1)},F_{\bullet 0 (b,b+1)},F_{1 \bullet (b,b+1)}$ as $a$ ranges over all odd numbers in $[n]$ and $b$ ranges over all even numbers in 
    %$\{0\} \cup [n]$
    $[\overline{n}]$. See Figure \ref{fig:alternating-hook}.
\end{definition}

 \begin{figure}[H]
        \centering
        \includegraphics[width=0.4\linewidth]{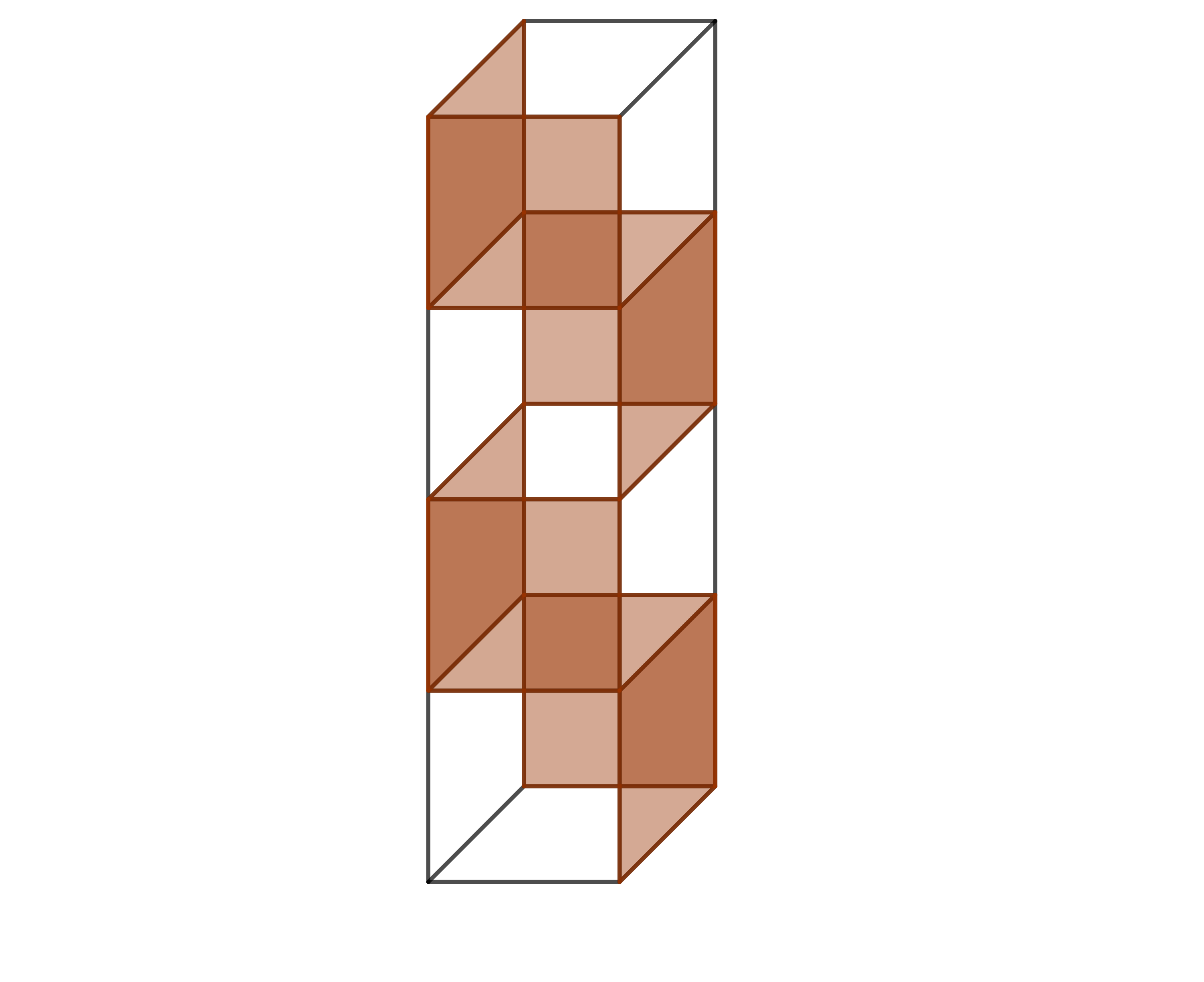}
        \caption{The alternating hook minors $\mathrm{altH}(4)$}
        \label{fig:alternating-hook}
    \end{figure}

The 
alternating hook minors of $E_A$ enjoy the following nice property.

\begin{proposition} \label{prop: hooks_nonvanishing}
   For any subset $S \subset \mathrm{altH}(n) \cup \{F_{\bullet \bullet n} \}$,
  if $W \in (\CC^*)^{4n+4}$  is chosen to be a generic solution to the system given by $S$,
   then no other factor of $E_A$ will vanish at $W$.  
\end{proposition}

\begin{proof}
   Since $S$ only contains $2 \times 2$ determinant factors of $E_A$ and it contains no cups as well as no overlapping minors on a face of $W$, no other $2 \times 2$ determinant factors of $E_A$ vanishes. 
   Furthermore, since for any choice of distinct $i,j \in [\overline{n}]$, the subtensor given by the coefficients of $f_i,f_j$ makes at  most a pair of minors in that subtensor vanish, a generic $W$ cannot cause any $2 \times 2 \times 2$ hyperdeterminantal factor of $E_A$ to vanish.
   Lastly, 
   for any pairwise distinct $i,j,k \in [\overline{n}]$, by Lemma \ref{lem: minors H}, we do not have enough minors in the corresponding $2 \times 2 \times 3$ subtensor vanish that would cause the hyperdeterminant $H_{ijk}$ to vanish. 
\end{proof}

A consequence of the above result is that we can use Corollary \ref{cor: point_formula} to compute  $\mldeg(X_{W,n})$ when $W$ satisfies the conditions stipulated by Proposition \ref{prop: hooks_nonvanishing}.

\begin{lemma}\label{lem: hook_formula}
For any subset $S \subset \mathrm{altH}(n) \cup \{ F_{\bullet \bullet n} \}$,
  if $W \in (\CC^*)^{4n+4}$  is chosen to be a generic solution to the system given by $S$,
 then $\mldeg(X_{W,n}) = (n+2)(n+1) - \mid S\mid$.
\end{lemma}

\begin{proof}
    We prove this statement by directly counting intersection points and summing Euler characteristics according to Corollary \ref{cor: point_formula}.
    If $S$ contains $F_{\bullet \bullet n}$, then $\sum_{k=0}^n \chi(Q_k) = -(2n+1)$, since exactly one quadric is singular. Meanwhile, for every $2$-minor in $S$, the pair of quadrics whose coefficients overlap with those of this minor will either have one of their two intersection points on an axis or share an asymptote, meaning that the number of pairwise intersection points among the quadrics will be $((n+1)n-(|S|-1))$.
    Corollary \ref{cor: point_formula} implies that $\mldeg(X_{W,n}) = (n+2)(n+1) - |S|$.
    If $S$ does not contain $F_{\bullet \bullet n}$, then  $\sum_{k=0}^n \chi(Q_k) = -(2n+2)$.
    By the same argument as for the first case, the number of pairwise intersection points among the quadrics  will be $((n +1)n-|S|)$. Therefore $\mldeg(X_{W,n}) = (n+2)(n+1)-|S|$.
\end{proof}

\begin{example}

    We give an example for $n=2$ where we pick $S = \mathrm{altH}(2) \cup \{ F_{\bullet \bullet 2} \}$. Note that $|S| = 5$. 
    We use $W$ with quadrics 
       \begin{align*}
     f_0 &= 2 + x + 5y + 2xy, \\
     f_1 &= 2 + 2x + 5y + 2xy, \\
     f_2 &= 1 + x + y + xy. 
    \end{align*} 
    We see that $Q_0$ and $Q_1$ are nonsingular while $Q_2$ is singular. Therefore, their Euler characteristics are $-2,-2,-1$ respectively.  Figure \ref{fig:three quadrics} shows the arrangement of these three quadrics with two pairwise intersection points in the torus. 
\begin{figure}[H]
        \centering
        \includegraphics[width=0.4\linewidth]{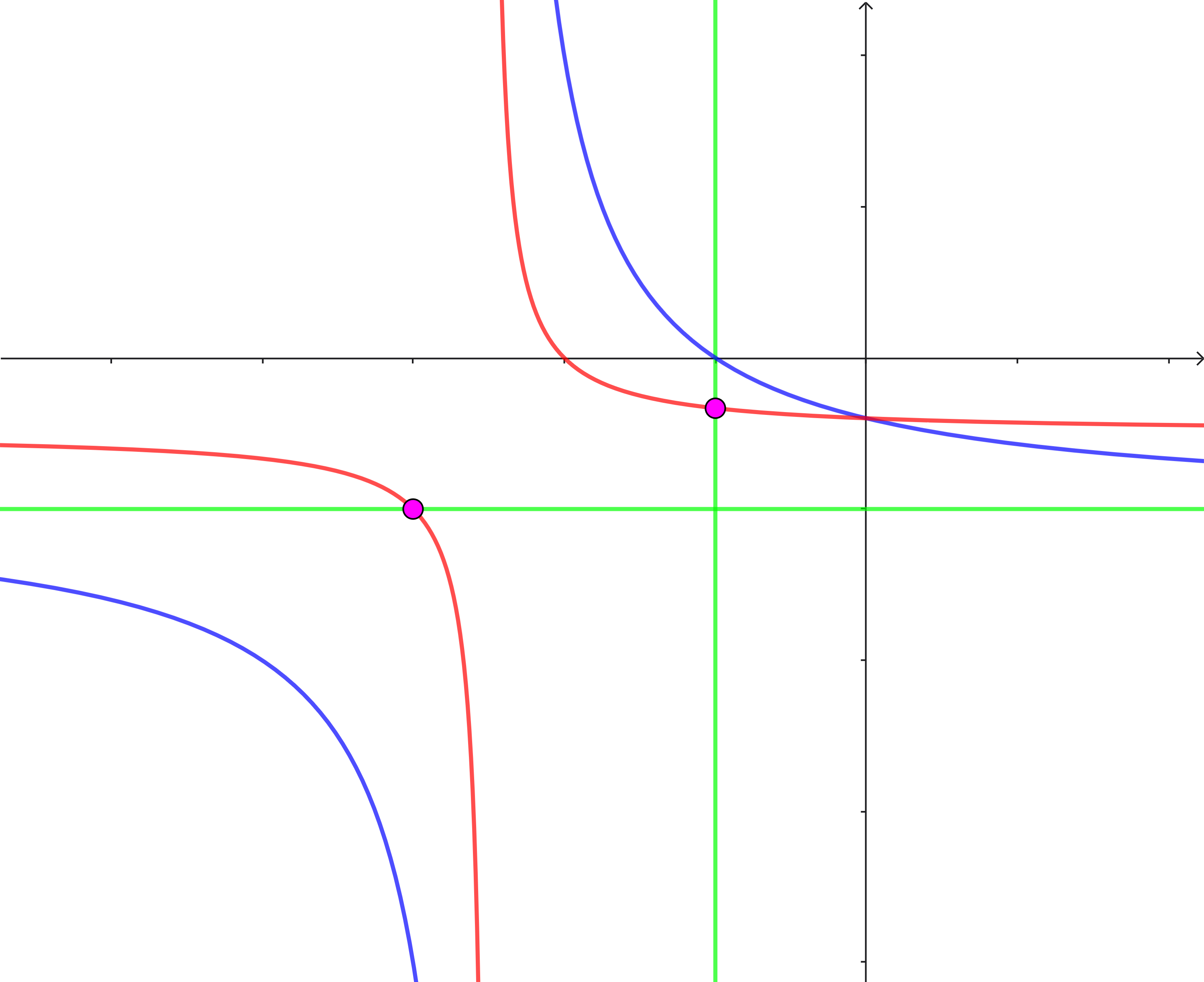}
        \caption{Three quadrics with $S= \mathrm{altH}(2) \cup F_{\bullet \bullet 2}$.}
        \label{fig:three quadrics}
    \end{figure}
\end{example}
\noindent We conclude that $\mldeg(X_{W,2}) = |-2-2-1-2| = 7$,
which is equal to $4 \cdot 3 - 5$.

\begin{definition}
    The feasible set $ \mathcal{A}_n \subset [(n+2)(n+1)]$ is the subset with the property that for every $\alpha \in \mathcal{A}_n$, there exists some $W \in (\CC^*)^{4n+4}$ such that $\mldeg(X_{W,n}) = \alpha$.
\end{definition}

%We make use of the following result to proceed towards answering the question of achievability.

\begin{lemma}\label{lem:inductive_euler_char}
    For all $n \geq 2$, $\mathcal{A}_{n-1} \subset \mathcal{A}_n$
\end{lemma}

\begin{proof}
    Consider the set of all $W$ such that $W_{\bullet \bullet (n-1)} = W_{\bullet \bullet n}$.
    Then for any choice of $W$ in this set, $Q_n = Q_{n-1}$ and therefore

    $$ \chi\left(Y_{W,n}\right) = (-1)^n \chi\left(V(f_0f_1\cdots f_n) \cap (\CC^*)^2\right) = (-1)^n\chi\left(V(f_0f_1\cdots f_{n-1}) \cap (\CC^*)^2\right).$$
\end{proof}

%\begin{theorem}
%    For any $n$, $\mathcal{A}_n = [(n+2)(n+1)]$.
%\end{theorem}

\begin{proof}[Proof of Theorem \ref{thm:realizability}:]
 We use induction on $n$, where the base case of $n=1$ is covered by Theorem \ref{thm:P1 x P1 x P1}.
   By Lemma \ref{lem:inductive_euler_char}, it suffices to show that $\{(n+1)n+1,\dots,(n+2)(n+1)\} \subset \mathcal{A}_n $.
    For any $\alpha \in \{(n+1)n+1,\dots,(n+2)(n+1)\}$, we pick $S \subset \mathrm{altH}(n) \cup \{F_{\bullet \bullet n}\} $ so that $ |S| = (n+2)(n+1) - \alpha$ and a scaling $W \in (\CC^*)^{4n+4}$ which is a generic solution to the system given by $S$. Then by Lemma \ref{lem: hook_formula}, $\mldeg(X_{W,n}) = (n+2)(n+1) - ((n+2)(n+1)-\alpha) = \alpha$.  
\end{proof}

\subsubsection*{Acknowledgements}
This project originated at MPI MiS Leipzig during the SLMath Summer School: New perspectives on discriminants and their applications.  The authors thank the organizers and the hosts. The fourth author thanks Taylor Brysiewicz for many helpful conversations. The second and fourth authors are funded by the Deutsche Forschungsgemeinschaft (DFG, German Research Foundation) – 539867500 and 539847176 respectively as part of the research priority program Combinatorial Synergies.

%%%%%%%%%%%%%%%%%%%%%%%%%%%%%%%%%%%%%%%%%%%%

\bibliographystyle{alpha}
\bibliography{bibliography2}

\vspace{0.5cm}

\noindent{\bf Authors' addresses:}
\smallskip
\small 

\noindent Serkan Hoşten,
San Francisco State University
\hfill {\tt serkan@sfsu.edu}

\noindent Vadym Kurylenko,
Otto-von-Guericke-Universität Magdeburg
\hfill {\tt vadym.kurylenko@ovgu.de}

\noindent Elke Neuhaus, MPI for Mathematics in the Sciences Leipzig 
\hfill {\tt elke.neuhaus@mis.mpg.de}

\noindent Nikolas Rieke,
Technische Universit\"at Braunschweig
\hfill {\tt nikolas.rieke@tu-braunschweig.de}

\end{document}